\documentclass[smallextended,referee,envcountsect]{svjour3}
\usepackage [latin1]{inputenc}
\usepackage{amsmath,amssymb}
\usepackage{marvosym}
\usepackage{graphicx}
\usepackage{subcaption}
\usepackage{lipsum}
\usepackage{color}
\usepackage{epstopdf}
\usepackage[numbers,sort&compress]{natbib}
\usepackage[colorlinks,linkcolor=blue,urlcolor=blue,citecolor=blue]{hyperref}

\smartqed
\usepackage{graphicx}
\journalname{}

\usepackage{fancyhdr}
\usepackage{ntheorem}
\theoremheaderfont{\bfseries\upshape}
\theorembodyfont{\upshape}
\renewtheorem{remark}{\it Remark}[section]
\renewtheorem{example}{Example}[section]

\begin{document}

\title{Tikhonov Regularization of Second-order plus First-order Primal-dual Dynamical Systems for Separable Convex Optimization}

\author{Xiangkai Sun$^ {1}$\and Lijuan Zheng$^{1}$ \and  Kok Lay Teo$^{2}$}

\institute{\\Xiangkai Sun (\Letter)  \at{\small sunxk@ctbu.edu.cn } \\
          \\Lijuan Zheng \at {\small zhenglljuanlzm@163.com} \\ \\Kok Lay Teo\at{\small K.L.Teo@curtin.edu.au}\\\\
               $^{1}$Chongqing Key Laboratory of Statistical Intelligent Computing and Monitoring, College of Mathematics and Statistics,
 Chongqing Technology and Business University,
Chongqing 400067, China\\
$^{2}$School of Mathematical Sciences, Sunway University, Bandar Sunway, 47500 Selangor Darul Ehsan, Malaysia.
}

\date{Received: date / Accepted: date}

\maketitle

\begin{abstract}
 This paper deals with a Tikhonov regularized second-order plus first-order primal-dual dynamical system with time scaling for separable convex optimization problems with linear equality constraints. This system consists of two second-order ordinary differential equations for the primal variables and one first-order ordinary differential equation for the dual variable. By utilizing the Lyapunov analysis approach, we obtain the convergence properties of the primal-dual gap, the objective function error, the feasibility measure and the gradient norm of the objective function along the trajectory. We also establish the strong convergence of the primal trajectory generated by the dynamical system towards the minimal norm solution of the separable convex optimization problem. Furthermore, we give numerical experiments to illustrate the theoretical results, showing that our dynamical system performs better than those in the literature in terms of convergence rates.
\end{abstract}

\section{Introduction}

Let $\mathcal{X}$ be a real Hilbert space equipped with the inner product $\langle\cdot,\cdot\rangle$ and the norm $\|\cdot\|$. Let $f:\mathcal{X}\rightarrow\mathbb{R}$ be a convex differentiable function. The unconstrained convex optimization problem is defined as follows:
\begin{eqnarray}\label{in1.1}
\min_{x\in\mathcal{X}}{f(x)}
\end{eqnarray}
In recent years, the second-order dynamics approach has attracted the interest of many researchers as one of the powerful frameworks for finding the solution set of problem (\ref{in1.1}). In order to solve problem (\ref{in1.1}), Polyak \cite{B1964} first introduces the heavy ball method with a friction system:
\begin{eqnarray*}
\ddot{x}( t )+\gamma \dot{x}( t )+\nabla f( x( t ))=0,
\end{eqnarray*}
where $\gamma>0$ is a constant damping coefficient. Su et al. \cite{SBC2016} propose the following inertial dynamical system with asymptotically vanishing damping:
\begin{eqnarray*}
~~~~~~~~~~~~~~~~~~~~~~~~~~~~~~~~~~~~~~~~\ddot{x}(t)+\frac{\alpha}{t}\dot{x}(t)+\nabla f(x(t))=0,~~~~~~~~~~~~~~~~~~~~~~~~~~~~~(\textup{AVD}_{\alpha})
\end{eqnarray*}
where $\alpha>0$ is a constant and $\frac{\alpha}{t}$ is the vanishing damping coefficient. In particular, in the case $\alpha=3$, the dynamical system (\textup{AVD}$_{\alpha}$) can be regarded as the continuous limit of the Nesterov's accelerated gradient algorithm \cite{y1983}. Moreover, when $\alpha\geq3$,  they show that as time $t$ approaches infinity, the asymptotic convergence rate of the objective function value along the trajectory generated by (\textup{AVD}$_{\alpha}$) is $\mathcal{O}\left ( \frac{1}{t^2}\right )$. In the last few years, the convergence properties of the inertial system (\textup{AVD}$_{\alpha}$) with general damping coefficients have been intensively investigated from several different perspectives, see \cite{ac2017,CEG2009,ceg2009,accr2018,mp2018att}.

It is well known that the time scaling technique is an efficient way to improve the convergence rate of (\textup{AVD}$_{\alpha}$).  Attouch et al. \cite{acr2019} give the following inertial dynamical system with vanishing damping and a time scaling coefficient:
\begin{eqnarray*}
\ddot{x}(t)+\frac{\alpha}{t}\dot{x}(t)+\beta(t)\nabla f(x(t))=0,
\end{eqnarray*}
where $\alpha >0$ is a constant, $t\geq t_0>0$, and $\beta:[t_0,+\infty)\rightarrow[0,+\infty)$ is the time scaling function. They establish fast convergence rate of the objective function error, which can be regarded as an extension of the results obtained in \cite{SBC2016}. For more convergence results on the unconstrained optimization problem (\ref{in1.1}) for inertial dynamical system with time scaling, see \cite{wwj2016,wrj2021,botc2022,bot2023,2024gopt}.

On the other hand, Tikhonov regularization technique ensure that the trajectory converges strongly to the minimal norm solution of the unconstrained optimization problem (\ref{in1.1}), rather than weakly to an arbitrary minimizer. In this context, many scholars have studied the Tikhonov regularized second-order dynamical system for the unconstrained optimization problem (\ref{in1.1}). For example, Attouch et al. \cite{ACr2018} introduce the following Tikhonov regularized second-order dynamical system:
\begin{eqnarray}\label{in1.2}
\ddot{x} ( t )+\frac{\alpha}{t}\dot{x}( t )+\nabla f( x ( t ) )+\epsilon ( t )x ( t )=0,
\end{eqnarray}
where $t\geq t_0$ and $\epsilon:[t_0,+\infty)\rightarrow[0,+\infty)$ is the Tikhonov regularization function. They show that the dynamical system (\ref{in1.2}) maintains a fast convergence rate of the objective function value along the trajectory. Furthermore, they find that each trajectory of the dynamical system (\ref{in1.2}) converges strongly to the minimal norm solution of the unconstrained optimization problem (\ref{in1.1}). For more details on the Tikhonov regularized second-order dynamical systems for the unconstrained optimization problem (\ref{in1.1}) and the more general monotone inclusion problem, see \cite{ACS2021,bcl2021,xw2021,abcr2023,L2023, jmaa2024bot,ck2024, ka2024}.

Note that all of the above papers focus on the study of the unconstrained optimization problem (\ref{in1.1}). Recently, various second-order dynamical systems in primal-dual framework have been introduced to deal with the following constrained convex optimization problem:
\begin{eqnarray}\label{in1.3}
\left\{ \begin{array}{ll}
&\mathop{\mbox{min}}\limits_{x\in\mathcal{X}}~~{f(x)}\\
&\mbox{s.t.}~~Ax=b,
\end{array}
\right.
\end{eqnarray}
where $\mathcal{X}$ and $\mathcal{Z}$ are real Hilbert spaces, $A:\mathcal{X}\rightarrow\mathcal{Z}$ is a continuous linear operator and $b\in\mathcal{Z}$. To solve problem (\ref{in1.3}), Zeng et al. \cite{ZLC2023} propose the following inertial primal-dual dynamical system with vanishing damping:
\begin{eqnarray*}
\left\{ \begin{array}{ll}
&\ddot{x}(t)+\frac{\alpha }{t}\dot{x}(t)+\nabla f(t)+A^\top\left ( \lambda (t)+\kappa t\dot{\lambda }(t)\right )+A^\top\left ( Ax(t)-b\right )=0,\\
&\ddot{\lambda }(t)+\frac{\alpha }{t}\dot{\lambda }(t)-\left(A( x (t)+\kappa t\dot{x }(t))-b\right)=0,
\end{array}
 \right.
\end{eqnarray*}
where $\alpha >3$ and $\kappa=\frac{1}{2}$. They show that the fast convergence rates of the primal-dual gap and the feasibility measure along the trajectory are $\mathcal{O}\left(\frac{1}{t^2}\right)$ and $\mathcal{O}\left(\frac{1}{t}\right)$, respectively. Subsequently, Bo\c{t} and Nguyen \cite{BN2021} improve the convergence rates of the work in \cite{ZLC2023} and obtain weak convergence results for the trajectory to the primal-dual optimal solution of   problem (\ref{in1.3}). Hulett and Ngugen \cite{hn2023} further study the inertial primal-dual dynamical system involving time scaling coefficient in \cite{ZLC2023}.

Recently, He et al. \cite{HHF2022} introduce the following ``second-order primal" plus ``first-order dual" inertial dynamical system with constant viscous damping and time scaling coefficients for problem (\ref{in1.3}):
\begin{eqnarray*}
\left\{ \begin{array}{ll}
&\ddot{x}(t)+\gamma \dot{x}(t)+\beta (t)\left(\nabla f(x(t))+A^\top\lambda (t)+\sigma A^\top\left ( Ax(t)-b\right )\right)=0,\\
&\dot{\lambda }(t)-\beta (t)\left ( A\left ( x(t)+\delta \dot{x}(t)\right )-b\right )=0.
\end{array}
 \right.
\end{eqnarray*}
Here, $t\geq t_0>0$, $\gamma>0$ is the constant damping coefficient, $\delta>0$ is the extrapolation coefficient, $\sigma\geq0$ is the penalty parameter of the corresponding augmented Lagrangian function, and $\beta:[t_0,+\infty)\rightarrow (0,+\infty)$ is the time scaling function. They show that the convergence rates of the objective function error and the feasibility measure can grow exponential when the time scaling grows exponentially. As an extension of the inertial dynamical system in \cite{HHF2022}, a second-order plus first-order primal-dual dynamical system with time scaling and vanishing damping for problem (\ref{in1.3}) is investigated in \cite{hhf2022}. Zhu et al. \cite{ZHF2024} introduce a Tikhonov regularized second-order plus first-order primal-dual dynamical system with asymptotically vanishing damping for problem (\ref{in1.3}), and establish the fast convergence rates of the primal-dual gap, the feasibility measure and the objective error along the trajectory. More results on the convergence rates of inertial dynamical systems for the linearly constrained convex optimization problem (\ref{in1.3}) can be found in \cite{htlf2023,hhf2024, zdx2024,zhf2024,hhf2023}.

As a special case of linear equality constrained optimization problems, separable convex optimization with linear equality constraints has been used widely in various fields, such as machine learning, image recovery, statistical learning and signal recovery \cite{bpcpe20113,gosr20142,baijc,llf20194,hexing}. The separable convex optimization problem with linear equality constraints is defined as follows:
\begin{eqnarray}\label{1.1}
\left\{ \begin{array}{ll}
&\mathop{\mbox{min}}\limits_{x\in \mathcal{X}, y\in \mathcal{Y}}~\varPhi  (x,y):=f(x) + g(y)\\
&~~~\mbox{s.t.}~~~~~Ax+By=b,
\end{array}
\right.
\end{eqnarray}
where $f: \mathcal{X}\rightarrow \mathbb{R}$ and $g:\mathcal{Y}\rightarrow\mathbb{R}$ are two smooth convex functions, $A:\mathcal{X}\rightarrow\mathcal{Z}$ and $B:\mathcal{Y}\rightarrow\mathcal{Z}$ are two linear continuous operators and $b\in  \mathcal{Z}$.

We observe that only a few papers in the literature devoted to studying second-order dynamical systems for solving the separable convex optimization problem (\ref{1.1}). More precisely, He et al. \cite{hhf2021} consider the following second-order dynamical system with general damping parameters and extrapolation coefficients for the separable optimization problem (\ref{1.1}):
\begin{eqnarray}\label{in1.5}
\left\{ \begin{array}{ll}
\ddot{x}(t)+\gamma(t)\dot{x}(t)+\nabla f(x(t))+A^\top (\lambda(t)+\delta(t)\dot{\lambda}(t))+A^\top( Ax(t)+B y(t)-b)=0,\\
\ddot{y}(t)+\gamma(t) \dot{y}(t)+\nabla g(y(t))+B^\top(\lambda(t)+\delta(t)\dot{\lambda}(t)+B^\top( Ax(t)+B y(t)-b)=0,\\
\ddot{\lambda }(t)+\gamma(t)\dot{\lambda}(t)-( A( x(t)+\delta(t) \dot{x}(t))+B(y(t)+\delta(t)\dot{y}(t))-b)=0,
\end{array}
 \right.
\end{eqnarray}
where $t\geq t_0>0$, $\gamma:[t_0,+\infty)\rightarrow[0,+\infty)$ is the viscous damping function, and $\delta:[t_0,+\infty)\rightarrow[0,+\infty)$ is the extrapolation function. Following the Lyapunov analysis method, they investigate the asymptotic convergence rate of the dynamical system (\ref{in1.5}) and show that the obtained results are robust under external perturbations. Attouch et al. \cite{acfr2022} give the following time-rescaled inertial augmented Lagrangian system with viscous damping, extrapolation and time scaling:
\begin{eqnarray}\label{in1.4}
\small{\left\{ \begin{array}{ll}
\ddot{x}(t)+\gamma(t)\dot{x}(t)+\beta(t)\left( \nabla f(x(t))+A^\top\left( \lambda(t)+a(t)\dot{\lambda}(t)+\mu( Ax(t)+B y(t)-b)\right)\right)=0,\\
\ddot{y}(t)+\gamma (t)\dot{y}(t)+\beta(t)\left( \nabla g(y(t))+B^\top\left(\lambda(t)+a(t)\dot{\lambda}(t)+\mu( Ax(t)+B y(t)-b)\right)\right)=0,\\
\ddot{\lambda }(t)+\gamma(t)\dot{\lambda}(t)-\beta(t)( A( x(t)+a(t) \dot{x}(t))+B(y(t)+a(t)\dot{y}(t))-b)=0,
\end{array}
 \right.}
\end{eqnarray}
where $\gamma(t)$ is the viscous damping parameter, $\beta(t)$ is the time scaling parameter, $a(t)$ is the extrapolation parameter and $\mu$ is the penalty parameter of the corresponding augmented Lagrangian function. They show that the second-order inertial system (\ref{in1.5}) has a provably fast convergence rate in solving problem (\ref{1.1}).

From the above, it can be seen that second-order time-continuous dynamical systems with fast convergence guarantees for solving the separable convex optimization problem (\ref{1.1}) have received less attention so far than unconstrained optimization problems. Moreover, to the best of our knowledge, there is no literature devoted to the study of the Tikhonov regularized second-order plus first-order primal-dual dynamical system for the separable convex optimization problem (\ref{1.1}). Motivated by the works \cite{HHF2022,hhf2021,acfr2022}, we will provide some new results for the following second-order plus first-order primal-dual dynamical system, which consists of two second-order ordinary differential equations for the primal variables and one first-order ordinary differential equation for the dual variable,
\begin{eqnarray}\label{1.3}
\left\{ \begin{array}{ll}
\ddot{x}(t)+\gamma \dot{x}(t)+\beta(t)\left( \nabla f(x(t))+A^\top \lambda(t)+A^\top( Ax(t)+B y(t)-b)+\epsilon(t)x(t)\right)=0,\\
\ddot{y}(t)+\gamma \dot{y}(t)+\beta(t)\left( \nabla g(y(t))+B^\top\lambda(t)+B^\top( Ax(t)+B y(t)-b)+\epsilon(t)y(t)\right)=0,\\
\dot{\lambda }(t)-\beta(t)\left( A( x(t)+\delta \dot{x}(t))+B(y(t)+\delta\dot{y}(t))-b\right)=0.
\end{array}
 \right.
\end{eqnarray}
Here, $t\geq t_0>0$, $\gamma$ is a constant damping coefficient, $\delta $ is a constant extrapolation coefficient, $ \beta:[t_0,+\infty)\rightarrow(0,+\infty)$ is the time scaling function, and $\epsilon:  [ t_0,+\infty )\rightarrow [0,+\infty)$ is the Tikhonov regularization function. In the sequel, we give the following assumptions: $\beta $ is a non-decreasing and continuously differentiable function, $\epsilon $ is continuously differentiable and non-increasing function with  $\lim\limits_{t\to+\infty}\epsilon(t)=0$. Note that the dynamical system (\ref{1.3}) involves the inertial terms only for the primal variables. Our contributions can be more specifically stated as follows:
\begin{itemize}
\item[{\rm (i)}] We introduce a second-order plus first-order primal-dual dynamical system with time scaling and Tikhonov regularization terms for the separable convex optimization problem (\ref{1.1}).
\item[{\rm (ii)}] When $\int_{t_0}^{+\infty}\beta(t)\epsilon (t)dt< +\infty$, we show that the convergence rates of the primal-dual gap, the objective function error and the feasibility measure are $\mathcal{O}\left(\frac{1}{\beta(t)}\right)$, $\mathcal{O}\left(\frac{1}{\sqrt{\beta(t)}}\right)$ and $\mathcal{O}\left(\frac{1}{\sqrt{\beta(t)}}\right)$, respectively. Under the assumptions $\int_{t_0}^{+\infty}\epsilon (t)dt< +\infty$ and $\lim\limits_{t\to+\infty}\beta (t)=+\infty$, we establish the minimal properties of the trajectories generated by the dynamical system (\ref{1.3}). Specially, when $\beta(t)=t^{r_1}$ and $\epsilon (t)=\frac{C}{t^{r_2}}$ with $1< r_2\leq r_1+1$, we prove that the convergence rate of the primal-dual gap depends on $r_2$.
\item[{\rm (iii)}] When $\lim\limits_{t\to+\infty}\beta (t)\epsilon (t)=+\infty$ and $\int_{t_0}^{+\infty}\epsilon (t)dt< +\infty$, we show that the trajectory generated by the dynamical system (\ref{1.3}) converges strongly to the minimal norm solution of the separable convex optimization problem (\ref{1.1}).
    \item[\rm(iv)] Through some numerical examples, we demonstrate that the dynamical system (\ref{1.3}) performs better than the dynamical systems (\ref{in1.5}) and (\ref{in1.4}) in terms of energy error and feasibility gap.
\end{itemize}

The rest of the paper is organized as follows. In Sections 2, we recall some basic notions and present some preliminary results. In Section 3, we show the existence and uniqueness of the solution trajectory of the dynamical system (\ref{1.3}). In Section 4, we establish the convergence rates of the primal-dual gap, the objective function error, the feasibility measure and the gradient norm of the objective function along the trajectory generated by the dynamical system (\ref{1.3}). Furthermore, we also obtain the minimal properties of the trajectory generated by the dynamical system (\ref{1.3}). In Section 5, we show that the primal trajectory converges strongly to the minimal norm solution of problem (\ref{1.1}). In Section 6, we give numerical experiments to illustrate the theoretical results.

\section{Preliminaries}

Throughout this paper, let $\mathcal{X}$, $\mathcal{Y}$ and $\mathcal{Z}$ be real Hilbert spaces equipped with the inner product $ \langle\cdot,\cdot \rangle$ and the norm $\|\cdot\|$.  For any $x\in \mathcal{X}$ and $y\in \mathcal{Y}$, the norm of the Cartesian product $\mathcal{X}\times\mathcal{Y}$ is defined by
$$\| ( x,y )\|=\sqrt{\| x \|^2+ \| y \|^2}.$$
The closed ball centered at $0\in \mathcal{X}$ with radius $r > 0$ is denoted as $\mathbb{B}(0,r) :=\{x \in \mathcal{X }: \|x\|\leq r\}$,  and the notation ${L}_{loc}^{1}  ([ t_0,+\infty ))$ denotes the family of locally integrable functions. Let $\varphi:\mathcal{X}\rightarrow \mathbb{R}$ be a real-valued function. For any $\sigma>0$, we say that $\varphi$ is a $\sigma $-strongly convex function iff $ \varphi(\cdot)-\frac{\sigma}{2}\|\cdot\|$ is a convex function. Moreover, we say that the gradient of $\varphi$ is Lipschitz continuous on $\mathcal{X}$ iff there exists $0<l<+\infty$ such that
$$
\|\nabla \varphi (x_1)-\nabla \varphi (x_2)\|\leq l\|x_1-x_2\|,~~~~\forall x_1,x_2 \in \mathcal{X}.
$$
The Lagrangian function associated with problem  (\ref{1.1}) is defined as:
$$
L ( x,y,{\lambda})=f(x)+g(y )+\langle {\lambda} ,{Ax+By-b}  \rangle.
$$
Naturally, the augmented Lagrangian function associated with problem (\ref{1.1}) is defined as:
$$
\mathcal{L}( x,y,{\lambda})=f(x)+g(y)+  \langle {\lambda} ,{Ax+By-b} \rangle+\frac{1}{2}{\| Ax+By-b\|}^2.
$$
We denote the partial derivative of $\mathcal{L}$ with respect to the first argument by $\nabla _x\mathcal{L}$, and with respect to the second argument by $\nabla _y\mathcal{L}$. In addition, the associated Lagrange dual problem and the augmented Lagrange dual problem of  problem (\ref{1.1}) are:
\begin{equation*}
\max_{\lambda\in\mathcal{Z} }{\min_{x\in\mathcal{X},y\in\mathcal{Y}}{L(x,y,\lambda)}}
\end{equation*}
and
\begin{equation}\label{1.6}
\max_{\lambda\in\mathcal{Z} }{\min_{x\in\mathcal{X},y\in\mathcal{Y}}{\mathcal{L}( x,y,\lambda)}}.
\end{equation}
Denoted by $\Omega$ the saddle point set of problem $(\ref{1.1})$. Then, $( x^*,y^*,{\lambda}^* )\in \Omega$ if and only if
\begin{eqnarray}\label{1.2}
\left\{ \begin{array}{ll}
-A^\top\lambda^*=\nabla f( x^*),\\
-B^\top\lambda^*=\nabla g( y^*),\\
Ax^*+By^*-b=0.
\end{array}
 \right.
\end{eqnarray}
Obviously, $(\ref{1.2})$ can be rewritten as:
\begin{eqnarray}\label{1.5}
\left\{ \begin{array}{ll}
\nabla f( x^*)+A^\top\lambda^*+A^T( Ax^*+By^*-b)=0,\\
\nabla g( y^*)+B^\top\lambda^*+B^T( Ax^*+By^*-b)=0,\\
Ax^*+By^*-b=0.
\end{array}
 \right.
\end{eqnarray}
This means that
\begin{eqnarray*}
\left\{ \begin{array}{ll}
\nabla _x\mathcal{L}( x^*,y^*,\lambda^* )=0,\\
\nabla _y\mathcal{L}( x^*,y^*,\lambda^*  )=0,\\
\nabla _\lambda \mathcal{L}(x^*,y^*,\lambda^*)=0.
\end{array}
 \right.
\end{eqnarray*}
Thus, $\mathcal{L}( x,y,\lambda)$ and  $L(x,y,\lambda)$ have  the same set $\Omega$ of saddle points.

In what follows, we assume that $\Omega \neq\emptyset$. Then, the solution set $S$ of problem (\ref{1.1}) is nonempty. It is well-known that $( x^*,y^*)\in S$ if and only if there exists a solution ${\lambda} ^*$ of problem $(\ref{1.6})$ such that $( x^*,y^*,{\lambda} ^*)\in\Omega $. This means that
$(x^*,y^*,{\lambda} ^*)\in\Omega $ if and only if\\
\begin{equation}\label{5555}
\mathcal{L}( x^*,y^*,{\lambda})\leqslant \mathcal{L}( x^*,y^*,{\lambda}^*)\leqslant \mathcal{L}( x,y,{\lambda} ^*),~\forall (x,y,{\lambda})\in \mathcal{X}\times\mathcal{Y}\times \mathcal{Z}.
\end{equation}
For any $\lambda$, we set
\begin{equation*}
d( \lambda ):=\min_{x\in\mathcal{X},y\in\mathcal{Y}}{\mathcal{L}( x,y,\lambda)}
~~~
\mbox{and}
~~~
D( \lambda ):=\mathop{\textup{argmin}}\limits_{x\in\mathcal{X},y\in\mathcal{Y}}{\mathcal{L}( x,y,\lambda)}.
\end{equation*}
Obviously, $d(\lambda)$ is differentiable (see \cite[Chapter III: Remark 2.5]{fg1983}) and
\begin{equation*}
\nabla d( \lambda)=Ax( \lambda)+By( \lambda )-b,
\end{equation*}
where $(x(\lambda), y(\lambda))\in D(\lambda)$. Moreover, if $\lambda^*$ is a solution of problem $(\ref{1.6})$, then $$Ax( \lambda ^*)+By( \lambda^*)=b , ~~ \forall (x(\lambda^*),y(\lambda^*))\in D(\lambda^*).$$ This, together with $(\ref{1.5})$, gives
\begin{equation}\label{1.4}
D ( \lambda ^* )=  \left\{  ( x,y )\in \mathcal{X}\times\mathcal{Y}~|~( x,y )\in\mathop{\textup{argmin}}\limits_{x\in\mathcal{X},y\in\mathcal{Y}}{\mathcal{L}( x,y,\lambda)}\right\}\subseteq S.
\end{equation}

The following important property will be used in the sequel.
\begin{lemma}\label{lem2.1}\textup{\cite[Lemma A.3.]{ACr2018}}
Suppose that $\delta>0$, $\phi \in {L}^1([ \delta ,+\infty))$ is a nonnegative and continuous  function, and $\psi :[ \delta ,+\infty)\rightarrow (0,+\infty )$ is a nondecreasing function such that $\lim_{t\to+\infty}\psi (t )=+\infty$. Then,
\begin{equation*}
\lim_{t\to+\infty}\frac{1}{\psi(t)}\int_{\delta }^{t}\psi(s)\phi(s)ds=0.
\end{equation*}
\end{lemma}

\section{Existence and Uniqueness of the Solution Trajectory}

In this section, we establish the existence and uniqueness of the global solution of the dynamical system $(\ref{1.3})$.

\begin{theorem}
Suppose that $\nabla f$ is $l_1$-Lipschitz continuous on $\mathcal{X}$ with $l_1>0$, and $\nabla g$ is $l_2$-Lipschitz continuous on $\mathcal{Y}$ with $l_2>0$. Let $\epsilon(t)$, $\beta(t)$  and $\beta (t)\epsilon(t)\in{L}_{loc}^{1}([ t_0,+\infty))$. Then, for any given initial point $( x( t_0),y( t_0 ),\lambda  ( t_0),\dot{x}(t_0),\dot{y}( t_0))\in\mathcal{X}\times\mathcal{Y}\times\mathcal{Z}\times\mathcal{X}\times\mathcal{Y}$,
the dynamical system $(\ref{1.3})$ has a unique global solution.
\end{theorem}

\begin{proof}
Set $\mu(t):=\dot{x}(t)$, $\nu(t):=\dot{y}(t)$ and $Z(t):=(x(t),y(t),\lambda(t),\mu(t),\nu(t))$. Then, the dynamical system $(\ref{1.3})$ becomes
\begin{eqnarray*}
\left\{ \begin{array}{ll}
\dot{Z}(t)+G\left ( t,Z(t)\right )=0,\\
Z  ( t_0)=( x( t_0),y( t_0),\lambda( t_0),\mu ( t_0),\nu ( t_0)),
\end{array}
 \right.
\end{eqnarray*}
where
\begin{eqnarray}\label{exist0}\small{
G ( t,Z( t))=\begin{pmatrix}-\mu(t)
\\ -\nu(t)
\\ -\beta(t)(A(x(t)+\delta \mu(t))+B( y(t)+\delta\nu(t) )-b )
\\ \gamma \mu(t)+\beta(t)(\nabla f (x(t))+A^\top\lambda(t) +A^\top( Ax(t)+By(t)-b )+\epsilon(t)x(t) )
\\ \gamma \nu(t) +\beta(t)(\nabla g (y(t))+B^\top\lambda(t) +B^\top( Ax(t)+By(t)-b)+\epsilon(t)y(t))
\end{pmatrix}.}
\end{eqnarray}
Clearly, for any $Z(t)$ and $\bar{Z}(t)\in\mathcal{X}\times\mathcal{Y}\times\mathcal{Z}\times\mathcal{X}\times\mathcal{Y}$, we have
\begin{eqnarray}\label{exist1}
\begin{split}
\|G(t,Z(t))-G(t,\bar{Z}(t))\|\leq &(1+\gamma +\delta \beta(t)\| A\|)\|\mu(t)-\bar{\mu }(t)\|\\
&+(1+\gamma +\delta \beta(t)\|B\|)\| \nu(t) -\bar{\nu}(t)\|\\
&+\beta(t)\left(\| A^\top\|+\| B^\top\|\right)\|\lambda(t) -\bar{\lambda }(t)\|\\
&+\beta(t)\left(\| A\|+\| A^\top A\|+\| B^\top A\|+\epsilon(t)\right)\| x(t)-\bar{x}(t)\|\\
&+\beta(t)(\| B\|+ \| A^\top B \|+ \| B^\top B \|+\epsilon(t))\| y(t)-\bar{y}(t)\|\\
&+\beta(t)\| \nabla f(x(t))-\nabla f(\bar{x}(t))\|+\beta(t)\| \nabla g(y(t))-\nabla g(\bar{y}(t))\|.
\end{split}
\end{eqnarray}
Since $\nabla f$ is $l_1$-Lipschitz continuous  and $\nabla g$ is $l_2$-Lipschitz continuous, it follows from $(\ref{exist1})$ that
\begin{equation*}
\begin{split}
\| G ( t,Z(t))-G( t,\bar{Z}(t))\|
\leq &( 1+\gamma +\delta \beta (t)\| A\|)\| \mu(t)-\bar{\mu }(t)\|\\
&+( 1+\gamma +\delta \beta(t)\| B\|)\| \nu(t) -\bar{\nu }(t)\|\\
&+\beta(t)(\| A^\top\|+\| B^T\|)\| \lambda(t) -\bar{\lambda }(t)\|\\
&+\beta(t)(\| A\|+\| A^\top A\|+\| B^\top A\|+\epsilon(t))\| x(t)-\bar{x}(t)\|\\
&+\beta(t)(\| B\|+\| A^\top B\|+\| B^\top B\|+\epsilon(t))\| y(t)-\bar{y}(t)\|\\
&+\beta(t)l_1\| x(t)-\bar{x}(t)\|+\beta(t)l_2\| y(t)-\bar{y}(t)\|\\
\leq &( 2C_1+\delta \beta(t)(\| A\|+\|B\|)+3C_1\beta(t)+2\beta(t)\epsilon(t))\| Z(t)-\bar{Z}(t)\|,
\end{split}
\end{equation*}
where
\begin{equation*}
C_1:=\max \left\{1+\gamma ,  \|A^\top\|+  \|B^\top\| , \|A\|+ \|A^\top A\|+\|B^\top A\|+l_1,\|B\|+\|A^\top B\|+ \|B^\top B\|+l_2\right\}.
\end{equation*}
Set $$K(t):= 2C_1+\delta \beta(t) ( \| A\|+\|B\| )+3C_1\beta(t)+2\beta(t)\epsilon(t).$$ Clearly, $K(t)\in{L}_{loc}^{1}([ t_0,+\infty))$ and
\begin{equation*}
  \| G( t,Z(t) )-G ( t,\bar{Z}(t)) \|\leq K (t) \| Z(t)-\bar{Z}(t)\|.
\end{equation*}

On the other hand, it follows from $(\ref{exist0})$ that
\begin{eqnarray}\label{exist3}
\begin{split}
\| G( t,Z(t))\|\leq& (1+\gamma +\delta \beta(t) \| A \| ) \| \mu(t)\|+ (1+\gamma +\delta \beta(t) \|B\| ) \|\nu(t)\|\\
&+\beta(t)( \| A^\top \|+\| B^\top \|)\| \lambda(t) \|\\
&+\beta (t) ( \| A \|+ \| A^\top A\|+\| B^\top A\|+\epsilon(t)) \| x(t)\|\\
&+\beta(t)(\| B\|+\| A^\top B\|+  \| B^\top B  \|+\epsilon(t))  \| y(t)\|+\beta(t) \| b \|\\
&+\beta(t)(\| \nabla f( x(t)) \|+\|\nabla g ( y(t))\|+\| A^\top b\|+\| B^\top b\|).
\end{split}
\end{eqnarray}
Since $\nabla f$ is $l_1$-Lipschitz continuous and $\nabla g$ is $l_2$-Lipschitz continuous, we have
\begin{equation*}
\|\nabla f( x(t))-\nabla f(0)\|\leq l_1\| x(t) \|\leq l_1  \| Z(t)\|,
\end{equation*}
and
\begin{equation*}
\|\nabla g(y(t))-\nabla g(0)\|\leq l_2 \| y(t)\|\leq l_2\| Z(t)\|.
\end{equation*}
Combing these with  $(\ref{exist3})$, we have
\begin{eqnarray}\label{exist4}
\begin{split}
\| G( t,Z(t))\|\leq&( 1+\gamma +\delta \beta(t)\| A \|)\|Z(t)\|+ ( 1+\gamma +\delta \beta(t)\| B \| ) \| Z(t) \|\\
&+\beta(t) (\| A^\top\|+\| B^\top \|)\| Z(t)\|\\
&+\beta(t)(\| A\|+\| A^\top A\|+\| B^\top A\|+\epsilon(t))  \| Z(t) \|\\
&+\beta(t) (\| B\|+\| A^\top B\|+ \| B^\top B \|+\epsilon(t) ) \| Z(t) \|+\beta(t)\| b\|\\
&+\beta(t)(( l_1+l_2)\| Z(t)\|+ \|\nabla f(0)\|+\|\nabla g(0)\|+\| A^\top b\|+\| B^\top b\|)\\
\leq &( 1+\gamma +\delta \beta(t)\| A\|)( 1+\| Z(t)\|)
+(1+\gamma +\delta \beta(t)\| B\|)( 1+\| Z(t)\|)\\
&+\beta(t)( \| A^\top\|+\| B^\top \|)( 1+ \| Z(t)\|)\\
&+\beta(t)(\| A\|+ \| A^\top A \|+\| B^\top A \|+\epsilon(t)) ( 1+\| Z(t)\|)\\
&+\beta(t) (\| B \|+\| A^\top B \|+ \| B^\top B\|+\epsilon(t)) ( 1+\| Z(t)\| )\\
&+\beta(t)\| b \| ( 1+\| Z(t)\|)+C_2\beta(t)( 1+\| Z(t)\| ),
\end{split}
\end{eqnarray}
where
$$
C_2:=\max{\{ l_1+l_2,\|\nabla f(0)\|+\|\nabla g(0)\|+\| A^Tb\|+\| B^Tb\|\}}.$$
Now, let
$$C_3:=\max{ \{ 1+\gamma ,\| A^\top \|+\| B^\top\|,\| A\|+\| A^\top A\|+\| B^\top A\|,\| B\|+\| B^\top B\|+\| A^\top B \|,\|b\|,C_2\}}.
$$
Then, it follows from  $(\ref{exist4})$ that
\begin{equation*}
\begin{split}
\| G( t,Z(t))\|\leq &( C_3+\delta \beta(t)\| A \|)( 1+ \| Z(t) \| )+( C_3 +\delta \beta(t) \| B \| ) ( 1+ \| Z(t)\|)\\
&+3C_3\beta (t) ( 1+  \| Z(t)\| )+2\beta(t)( C_3+\epsilon(t))  ( 1+ \| Z(t)\| )\\
=&(2C_3+\delta\beta(t)(\|A\|+\|B\|)+5C_3\beta(t)+2\beta(t)\epsilon(t))( 1+\| Z(t)\|).
\end{split}
\end{equation*}
Moreover, let
$$
S(t):=2C_3+\delta\beta(t)(\|A\|+\|B\|)+5C_3\beta(t)+2\beta(t)\epsilon(t).
$$
Clearly, $S(t)\in{L}_{loc}^{1} ([ t_0,+\infty))$ and
\begin{equation*}
\|G(t,Z(t))\|\leq S(t)(1+ \|Z(t)\|),~~\forall Z(t)\in\mathcal{X}\times\mathcal{Y}\times\mathcal{Z}\times\mathcal{X}\times\mathcal{Y}.
\end{equation*}
According to \cite[Proposition 6.2.1]{ah1991}, for any given initial point $( x( t_0),y( t_0 ),\lambda  ( t_0),\dot{x}(t_0),\dot{y}( t_0))\in\mathcal{X}\times\mathcal{Y}\times\mathcal{Z}\times\mathcal{X}\times\mathcal{Y}$,
the dynamical system $(\ref{1.3})$ has a unique global solution. \qed
\end{proof}

\section{Asymptotic Properties of the Dynamical System (\ref{1.3})}

In this section, by constructing appropriate energy functions, we establish some asymptotic convergence properties of the primal-dual gap, the objective function error, the feasibility violation and the gradient norm of the objective function along the trajectory generated by the dynamical system ${(\ref{1.3})}$. In what follows, we suppose that $f$ and $g$ are continuously differentiable convex functions, $\nabla f$ is $l_1$-Lipschitz continuous on $\mathcal{X}$ with $l_1>0$ and $\nabla g$ is $l_2$-Lipschitz continuous on $\mathcal{Y}$ with $l_2>0$.
\begin{theorem}\label{the4.1}
Let $\epsilon : [ t_0,+\infty )\rightarrow [0,+\infty  )$ and $\beta : [ t_0,+\infty )\rightarrow [0,+\infty )$ such that
\begin{equation}\label{assump}
\int_{t_0}^{+\infty}\beta(t)\epsilon (t)dt< +\infty,~~\dot{\beta }(t)\leqslant \frac{1}{\delta }\beta (t),~~\frac{1}{\delta}< \gamma.
\end{equation}
Suppose that $( x,y,\lambda ):[ t_0,+\infty)\rightarrow\mathcal{X}\times\mathcal{Y}\times\mathcal{Z}$ is a solution of the dynamical system \textup{(\ref{1.3})}. Then, for any $( x^*,y^*,\lambda ^*)\in \Omega$, the trajectory $\left\{( x(t),y(t),\lambda(t))\right\}_{t\geq t_0}$ is bounded and the following statements are satisfied:
\begin{enumerate}
\item[{\rm (i)}] $($Convergence rate results$)$
\begin{eqnarray*}
\begin{split}
&\mathcal{L} ( x( t), y( t ),\lambda ^*)-\mathcal{L} ( x^*,y^*,\lambda ^* )=\mathcal{O}\left( \frac{1}{\beta ( t )} \right), ~\textup{as}~ t\rightarrow+\infty,\\
&\| \varPhi ( x(t),y(t))-\varPhi ( x^*,y^*)\|=\mathcal{O} \left( \frac{1}{\sqrt{\beta ( t )}} \right), ~\textup{as}~ t\rightarrow+\infty,\\
&\left \| Ax\left ( t\right )+B y\left ( t\right )-b\right \|=\mathcal{O}\left ( \frac{1}{\sqrt{\beta (t)}}\right ), ~\textup{as}~ t\rightarrow+\infty,\\
&\| \nabla f( x( t))-\nabla f( x^*)\|=\mathcal{O}\left ( \frac{1}{\sqrt{\beta (t)}}\right ), ~\textup{as} ~t\rightarrow+\infty,\\
&\| \nabla g ( y(t))-\nabla g( y^*)\|=\mathcal{O}\left ( \frac{1}{\sqrt{\beta (t)}}\right ), ~\textup{as} ~t\rightarrow+\infty.
\end{split}
\end{eqnarray*}
\item[{\rm (ii)}] $($Integral estimate results$)$
\begin{eqnarray*}
\begin{split}
&\int_{t_0}^{+\infty}\frac{\delta \gamma-1 }{\delta }\left(\| \dot{x}(t)\|^2+\| \dot{y}(t)\|^2\right)dt< +\infty,\\
&\int_{t_0}^{+\infty}\left ( \frac{\beta(t)}{\delta }-\dot{\beta }(t)\right )( \mathcal{L}( x(t),y(t),\lambda ^*)-\mathcal{L}( x^*,y^*,\lambda^*))dt< +\infty,\\
&\int_{t_0}^{+\infty}\frac{\beta (t)\epsilon(t)}{2\delta }\left( \| x(t)-x^*\|^2+\| y(t)-y^* \|^2\right)dt< +\infty,\\
&\int_{t_0}^{+\infty}\beta(t)\| Ax(t)+By(t)-b\|^2dt< +\infty.
\end{split}
\end{eqnarray*}
\end{enumerate}
\end{theorem}

\begin{proof} We first introduce the energy function $E:\left [ t_0,+\infty\right )\rightarrow\left [ 0,+\infty\right )$ as follows:
\begin{eqnarray}\label{3.7}
\begin{split}
E(t)=&~\beta(t)\left(\mathcal{L}(x(t),y(t),\lambda ^*)-\mathcal{L}(x^*,y^*,\lambda ^*)+\frac{\epsilon(t)}{2}\|x(t)\|^2+\frac{\epsilon(t)}{2}\|y(t)\|^2\right)\\
&+\frac{1}{2}\left\|\frac{1}{\delta}(x(t)-x^*)+\dot{x}(t)\right\|^2+\frac{\delta \gamma-1}{2\delta^2}\|x(t)-x^*\|^2\\
&+\frac{1}{2}\left\|\frac{1}{\delta}(y(t)-y^*)+\dot{y}(t)\right \|^2+\frac{\delta\gamma -1}{2\delta^2}\|y(t)-y^*\|^2+\frac{1}{2\delta}\left\|\lambda(t)-\lambda ^*\right\|^2.
\end{split}
\end{eqnarray}
Set
\begin{eqnarray*}
\left\{ \begin{array}{ll}
E _0(t):=\beta(t)(\mathcal{L}(x(t),y(t),\lambda ^*)-\mathcal{L}( x^*,y^*,\lambda ^*)),\\
E_1(t):=\frac{1}{2}\left \| \frac{1}{\delta }(x(t)-x^*)+\dot{x}(t)\right \|^2+\frac{\delta \gamma -1}{2\delta^2}\|x(t)-x^*\|^2+\frac{1}{2\delta }\| \lambda (t)-\lambda ^*\|^2+\frac{\beta(t)\epsilon (t)}{2}\| x(t)\|^2,\\
E_2(t):=\frac{1}{2}\left \| \frac{1}{\delta }( y(t)-y^*)+\dot{y}(t)\right \|^2+\frac{\delta \gamma -1}{2\delta^2}\| y(t )-y^*\|^2+\frac{1}{2\delta }\| \lambda (t)-\lambda ^*\|^2+\frac{\beta(t)\epsilon ( t)}{2}\| y( t)\|^2,\\
E_3( t):=-\frac{1}{2\delta}\| \lambda (t)-\lambda ^*\|^2.
\end{array}
 \right.
\end{eqnarray*}
Then,
\begin{eqnarray*}
E(t)=E_0(t)+E_1(t)+E_2(t)+E_3(t).
\end{eqnarray*}
Obviously,
\begin{eqnarray*}
\begin{split}
\dot{E}_1(t)=&~\left \langle \frac{1}{\delta }(x(t)-x^*)+\dot{x}(t),\frac{1}{\delta }\dot{x}(t)+\ddot{x}(t)\right \rangle+\frac{\delta \gamma -1}{\delta ^2}\langle x(t)-x^*,\dot{x}(t)\rangle\\
&+\frac{1}{\delta }\langle \lambda (t)-\lambda ^*,\dot{\lambda }(t)\rangle+\frac{1}{2}\left(\dot{\beta }( t)\epsilon (t)+\beta (t)\dot{\epsilon }( t)\right)\| x(t)\|^2+\beta(t)\epsilon(t)\langle x(t),\dot{x}(t)\rangle\\
=&~\left \langle \frac{1}{\delta }(x(t)-x^*)+\dot{x}(t),\left(\frac{1}{\delta }-\gamma \right )\dot{x}(t)-\beta(t)\big(\nabla _x\mathcal{L}(x(t),y(t),\lambda (t))+\epsilon (t)x(t)\big)\right\rangle\\
&+\frac{\delta \gamma -1}{\delta ^2}\langle x(t)-x^*,\dot{x}(t)\rangle+\frac{1}{\delta }\langle \lambda (t)-\lambda ^*,\dot{\lambda }(t)\rangle+\frac{1}{2}\left(\dot{\beta }(t)\epsilon (t)+\beta(t)\dot{\epsilon }(t)\right )\| x(t) \|^2\\
&+\beta(t)\epsilon(t)\langle x(t),\dot{x}(t)\rangle\\
=&~\left ( \frac{1}{\delta }-\gamma \right )\| \dot{x}(t)\|^2-\beta(t)\left \langle \frac{1}{\delta }(x(t)-x^*)+\dot{x}(t),\nabla_x\mathcal{L}( x(t),y(t),\lambda(t))\right \rangle\\
&-\beta(t)\left \langle \frac{1}{\delta }(x(t)-x^*)+\dot{x}(t),\epsilon(t)x(t)\right \rangle\\
&+\frac{\beta(t)}{\delta }\langle \lambda (t)-\lambda ^*,A\big(x(t)-x^*+\delta \dot{x}(t)\big)+B\big(y(t)-y^*+\delta \dot{y}(t)\big)\rangle\\
&+\frac{1}{2}\left ( \dot{\beta }(t)\epsilon(t)+\beta(t)\dot{\epsilon }(t)\right) \| x(t)\|^2+\beta (t)\epsilon (t)\langle x(t),\dot{x}(t)\rangle,
\end{split}
\end{eqnarray*}
where the second equality holds due to $(\ref{1.3})$, and  the last  equality follows from $(\ref{1.3})$ and $Ax^*+By^*=b$. Combining this with $\nabla _x\mathcal{L}( x(t),y(t),\lambda(t))=\nabla_x\mathcal{L}(x(t),y(t),\lambda ^*)+A^\top(\lambda(t)-\lambda ^*)$, we have
\begin{eqnarray}\label{3.2}
\begin{split}
\dot{E}_1(t)=&~\frac{1-\delta \gamma }{\delta }\| \dot{x}(t)\|^2-\beta (t)\left \langle \frac{1}{\delta }( x(t)-x^*)+\dot{x}(t),\nabla_x\mathcal{L}( x(t),y(t),\lambda ^*)\right \rangle\\
&-\beta (t)\left \langle \frac{1}{\delta }(x(t)-x^*),\epsilon(t)x(t)\right\rangle+\frac{\beta(t)}{\delta }\left \langle \lambda (t)-\lambda ^*,B\big( y(t)-y^*+\delta \dot{y}(t)\big)\right \rangle\\
&+\frac{1}{2}\left ( \dot{\beta }(t)\epsilon (t)+\beta (t)\dot{\epsilon }(t)\right )\| x(t) \|^2.
\end{split}
\end{eqnarray}
Similarly, it is easy to show that
\begin{eqnarray}\label{3.3}
\begin{split}
\dot{E}_2(t)=&~\frac{1-\delta \gamma }{\delta }\| \dot{y}(t) \|^2-\beta (t)\left \langle \frac{1}{\delta }( y(t)-y^* )+\dot{y}(t),\nabla_y\mathcal{L}( x(t),y(t),\lambda ^*)\right \rangle\\
&-\beta (t)\left \langle \frac{1}{\delta }( y(t)-y^*),\epsilon (t)y(t)\right\rangle+\frac{\beta (t)}{\delta } \langle \lambda (t)-\lambda ^*,A\big( x(t)-x^*+\delta \dot{x}(t)\big)\rangle\\
&+\frac{1}{2}\left ( \dot{\beta }(t)\epsilon(t)+\beta(t)\dot{\epsilon }(t)\right)\| y(t)\|^2.
\end{split}
\end{eqnarray}
Since $f+\frac{\epsilon (t)}{2} \| \cdot\|^2$ is an $\epsilon (t)$-strongly convex function, we get
\begin{eqnarray}\label{33convex}
\begin {split}
&f( x(t))+\frac{\epsilon (t)}{2} \| x(t)\|^2-f( x^*)-\frac{\epsilon(t)}{2}\| x^*\|^2\\
\leq& ~\langle \nabla f( x(t))+\epsilon(t)x(t),x(t)-x^* \rangle-\frac{\epsilon (t)}{2} \| x(t)-x^*\|^2.
\end{split}
\end{eqnarray}
By $(\ref{33convex})$ and $\nabla _x\mathcal{L}( x(t),y(t),\lambda ^*)=\nabla f( x(t))+A^\top\lambda ^*+A^\top\big( Ax(t)+By(t)-b\big)$, we have
\begin{eqnarray}\label{123}
\begin{split}
&\left\langle x(t)-x^*,\nabla _x\mathcal{L}( x(t),y(t),\lambda ^*)\right\rangle\\
=&~\langle x(t)-x^*,\nabla f( x(t))\rangle+\langle \lambda ^*,A( x(t)-x^*)\rangle+ \langle A( x(t)-x^*),Ax(t)+By(t)-b \rangle\\
\geq &~f( x(t))-f( x^*)+\frac{\epsilon (t)}{2}\left( \| x(t)\|^2-\| x^*\|^2+ \| x(t)-x^*\|^2\right )-\epsilon (t)\langle x(t),x(t)-x^*\rangle\\
&+\langle \lambda ^*,A( x(t)-x^* )\rangle+ \langle A( x(t)-x^* ),Ax(t)+By(t)-b\rangle.
\end{split}
\end{eqnarray}
Using a similar argument, we can show that
\begin{eqnarray}\label{1231}
\begin{split}
&\langle y(t)-y^*,\nabla _y\mathcal{L}( x(t),y(t),\lambda ^*)\rangle\\
\geq &~~g( y(t) )-g( y^* )+\frac{\epsilon (t)}{2}( \| y(t) \|^2- \| y^*\|^2+ \| y(t)-y^*\|^2 )-\epsilon(t) \langle y(t),y(t)-y^*\rangle\\
&+ \langle \lambda ^*,B( y(t)-y^* )\rangle+\langle B( y(t)-y^*),Ax(t)+By(t)-b \rangle.
\end{split}
\end{eqnarray}
Combining $(\ref{123})$, $(\ref{1231})$ and $Ax^*+By^*=b$, we have
\begin{eqnarray*}
\begin{split}
&\langle x(t)-x^*,\nabla _x\mathcal{L}( x(t),y(t),\lambda ^*) \rangle+\langle y(t)-y^*,\nabla _y\mathcal{L}( x(t),y(t),\lambda ^*) \rangle\\
\geq&~\mathcal{L}( x(t),y(t),\lambda ^*)-\mathcal{L}( x^*,y^*,\lambda ^*)+\frac{1}{2}\| Ax(t)+By(t)-b \|^2\\
&+\frac{\epsilon (t)}{2}(\| x(t) \|^2- \| x^* \|^2+\| x(t)-x^*\|^2)-\epsilon(t) \langle x(t),x(t)-x^*\rangle\\
&+\frac{\epsilon (t)}{2}(\| y(t) \|^2-\| y^* \|^2+\| y(t)-y^*\|^2)-\epsilon (t)\langle y(t),y(t)-y^* \rangle.
\end{split}
\end{eqnarray*}
This, together with (\ref{3.2}) and (\ref{3.3}), implies
\begin{eqnarray}\label{3.4}
\begin{split}
&\dot{E}_1(t)+\dot{E}_2(t)\\
\leq&~\frac{1-\delta \gamma }{\delta }\left ( \| \dot{x}(t)\|^2+ \| \dot{y}(t) \|^2\right )-\frac{\beta(t)}{\delta } ( \mathcal{L}( x(t),y(t),\lambda ^*)-\mathcal{L}( x^*,y^*,\lambda ^*))\\
&-\frac{\beta (t)}{2\delta }\| Ax(t))+By(t)-b\|^2\\
&-\frac{\beta (t)\epsilon (t)}{2\delta }\left( \| x(t)\|^2-\| x^*\|^2+\| x(t)-x^* \|^2+\| y(t)\|^2- \| y^*\|^2+\| y(t)-y^* \|^2 \right)\\
&-\beta(t)\left\langle \dot{x}(t),\nabla _x\mathcal{L}( x(t),y(t),\lambda ^* )\right\rangle-\beta(t) \langle \dot{y}(t),\nabla _y\mathcal{L}( x(t),y(t),\lambda ^* ) \rangle\\
&+\frac{\beta (t)}{\delta }\langle \lambda(t)-\lambda ^*,A\big( x(t)+\delta \dot{x}(t)\big )+B \big( y(t)+\delta \dot{y}(t)\big)-b\rangle\\
&+\frac{1}{2}\left( \dot{\beta }(t)\epsilon (t)+\beta(t)\dot{\epsilon }(t)\right)\left (  \| x(t)\|^2+\| y(t) \|^2\right ).
\end{split}
\end{eqnarray}
On the other hand,
\begin{eqnarray}\label{300}
\begin{split}
\dot{E}_0(t)=&~\dot{\beta } (t) ( \mathcal{L}( x(t),y(t),\lambda ^*)-\mathcal{L}( x^*,y^*,\lambda ^* ))\\
&+\beta (t)\langle \nabla _x\mathcal{L}( x(t),y(t),\lambda ^*),\dot{x}(t)\rangle+\beta (t)\langle \nabla _y\mathcal{L}( x(t),y(t),\lambda ^*),\dot{y} (t)\rangle
\end{split}
\end{eqnarray}
and
\begin{eqnarray}\label{3.5}
\begin{split}
\dot{E}_3(t)=&~-\frac{1}{\delta}\left\langle\lambda(t)-\lambda^*,\dot{\lambda}(t)\right\rangle\\
=&~-\frac{\beta(t)}{\delta } \langle \lambda(t)-\lambda ^*,A\big( x(t)+\delta \dot{x}(t)\big)+B\big( y(t)+\delta \dot{y}(t)\big)-b \rangle.
\end{split}
\end{eqnarray}
By (\ref{3.4}), (\ref{300}) and (\ref{3.5}), we have
\begin{small}
\begin{eqnarray}\label{et}
\begin{split}
\dot{E}(t)\leq &~\frac{1-\delta \gamma }{\delta }\left ( \| \dot{x}(t) \|^2+ \| \dot{y}(t)\|^2\right )+\left ( \dot{\beta }(t)-\frac{\beta(t)}{\delta }\right )( \mathcal{L}( x(t),y(t),\lambda ^*)-\mathcal{L}(x^*,y^*,\lambda^*))\\
&-\frac{\beta (t)}{2\delta } \| Ax(t)+By(t)-b \|^2+\frac{1}{2}\left ( \dot{\beta }(t)\epsilon(t)+\beta(t)\dot{\epsilon }(t)\right)( \| x(t)\|^2+ \|y(t)\|^2)\\
&-\frac{\beta(t)\epsilon(t)}{2\delta }\left ( \| x(t)\|^2-\| x^* \|^2+\| x(t)-x^* \|^2+\|y(t)\|^2- \| y^*\|^2+\| y(t)-y^*\|^2\right )\\
\leq&~ \frac{1-\delta \gamma }{\delta }\left ( \| \dot{x}(t)\|^2+ \| \dot{y}(t)\|^2\right )+\left ( \dot{\beta }(t)-\frac{\beta (t)}{\delta }\right )( \mathcal{L} ( x(t),y(t),\lambda ^* )-\mathcal{L}( x^*,y^*,\lambda ^*))\\
&-\frac{\beta (t)}{2\delta }\| Ax(t)+By(t)-b\|^2-\frac{\beta (t)\epsilon (t)}{2\delta }\left ( -\| x^*\|^2+\| x(t)-x^*\|^2-\| y^* \|^2+\| y(t)-y^* \|^2\right ),
\end{split}
\end{eqnarray}
\end{small}
where the last inequality is derived from (\ref{assump}).
It follows from (\ref{et}) that
\begin{eqnarray*}
\begin{split}
&\dot{E}(t)-\frac{1-\delta \gamma }{\delta } ( \| \dot{x}(t) \|^2+ \|\dot{y}(t) \|^2)-\left ( \dot{\beta }(t)-\frac{\beta (t)}{\delta }\right )( \mathcal{L}( x(t),y(t),\lambda ^* )-\mathcal{L}( x^*,y^*,\lambda ^*))\\
&+\frac{\beta(t)\epsilon (t)}{2\delta } ( \| x(t)-x^*\|^2+\| y(t)-y^* \|^2 )+\frac{\beta (t)}{2\delta }\| Ax(t)+By(t)-b\|^2\\
\leq&~\frac{\beta (t)\epsilon(t)}{2\delta } ( \| x^* \|^2+ \|y^* \|^2 ).
\end{split}
\end{eqnarray*}
Integrating the last inequality from $t_0$ to $t$, we obtain
\begin{eqnarray*}
\begin{split}
&E(t)+\int_{t_0}^{t}\frac{\delta \gamma-1 }{\delta }\left ( \| \dot{x}(s) \|^2+\| \dot{y}(s)\|^2\right )ds\\
&+\int_{t_0}^{t}\left ( \frac{\beta (s)}{\delta }-\dot{\beta }(s)\right )( \mathcal{L}( x(s),y(s),\lambda ^*)-\mathcal{L}( x^*,y^*,\lambda ^*))ds\\
&+\int_{t_0}^{t}\frac{\beta (s)\epsilon (s)}{2\delta }(\| x(s)-x^*\|^2+\| y(s)-y^*\|^2 )ds+\int_{t_0}^{t}\frac{\beta(s)}{2\delta }\| Ax(s)+By(s)-b \|^2 ds\\
\leq &~E( t_0)+\frac{\left (  \| x^* \|^2+ \|y^*\|^2\right ) }{2\delta }\int_{t_0}^{t}\beta (s)\epsilon (s)ds.
\end{split}
\end{eqnarray*}
Combining this with (\ref{assump}) and noting that $\mathcal{L} ( x(t),y(t),\lambda ^* )-\mathcal{L}( x^*,y^*,\lambda ^*)\geq 0$, it follows that $ \left\{ E(t) \right\}_{t\geq t_0}$ is bounded, and that
\begin{eqnarray*}
\int_{t_0}^{+\infty}\frac{\delta \gamma-1 }{\delta }\left(\| \dot{x}(t)\|^2+\| \dot{y}(t)\|^2\right)dt< +\infty,
\end{eqnarray*}
\begin{eqnarray*}
\int_{t_0}^{+\infty}\left ( \frac{\beta (t)}{\delta }-\dot{\beta }(t)\right )( \mathcal{L}( x(t),y(t),\lambda ^* )-\mathcal{L}(x^*,y^*,\lambda^* ) )dt< +\infty,
\end{eqnarray*}
\begin{eqnarray}\label{inte1}
\int_{t_0}^{+\infty}\frac{\beta (t)\epsilon(t)}{2\delta }\left( \| x(t)-x^*\|^2+\| y(t)-y^* \|^2\right)dt< +\infty,
\end{eqnarray}
and
\begin{eqnarray*}
\int_{t_0}^{+\infty}\frac{\beta(t)}{2\delta }\| Ax(t)+By(t)-b \|^2 dt<+\infty.
\end{eqnarray*}
Hence, the integral estimate results of $\mbox{(ii)}$ are satisfied.

Now, we derive  the convergence rate results of $\mbox{(i)}$.  In fact, by virtue of (\ref{3.7}) and the boundedness of $\left\{ E(t)\right\}_{t\geq t_0}$, it is clear that the trajectory $\left\{( x(t),y(t),\lambda(t))\right\}_{t\geq t_0}$ is bounded and
\begin{eqnarray}\label{3.8}
\mathcal{L} ( x(t), y(t),\lambda ^* )-\mathcal{L} ( x^*,y^*,\lambda ^* )=\mathcal{O}\left ( \frac{1}{\beta (t)}\right ), ~\textup{as}~ t\rightarrow+\infty.
\end{eqnarray}
Since $f$ and $g$ are convex, and $\nabla f$ is $l_1$-Lispschitz continuous and $\nabla g$ is $l_2$-Lispschitz continuous, it follows from \cite[Theorem 2.1.5]{n2004} that
\begin{eqnarray*}
f( x(t))-f( x^*)-\langle \nabla f ( x^*),x(t)-x^* \rangle\geq \frac{1}{2l_1} \| \nabla f( x(t) )-\nabla f( x^*)\|^2
\end{eqnarray*}
and
\begin{eqnarray*}
g( y(t))-g( y^* )- \langle \nabla g ( y^* ),y(t)-y^*\rangle\geq \frac{1}{2l_2}\| \nabla g( y(t) )-\nabla g( y^* )\|^2.
\end{eqnarray*}
Thus,
\begin{eqnarray*}
\begin{split}
&\mathcal{L}( x(t), y(t),\lambda ^*)-\mathcal{L}( x^*,y^*,\lambda ^*)\\
\geq&~ f ( x(t) )-f ( x^*)+ \langle A^\top\lambda ^*,x(t)-x^*\rangle+g( y(t))-g( y^*)+\langle B^\top\lambda ^*,y(t)-y^*\rangle\\
\geq &~\langle \nabla f ( x^* ),x(t)-x^* \rangle+\frac{1}{2l_1} \| \nabla f ( x(t) )-\nabla f ( x^*) \|^2+\langle A^\top\lambda ^*,x(t)-x^*\rangle\\
&+ \langle \nabla g( y^*),y(t)-y^*\rangle+\frac{1}{2l_2}\| \nabla g ( y(t) )-\nabla g ( y^* ) \|^2+ \langle B^\top\lambda ^*,y(t)-y^* \rangle\\
=&~\frac{1}{2l_1}\| \nabla f( x(t))-\nabla f( x^*)\|^2+\frac{1}{2l_2}\| \nabla g ( y(t) )-\nabla g( y^*) \|^2,
\end{split}
\end{eqnarray*}
where the last equality follows from (\ref{1.2}). Thus, it follows from (\ref{3.8}) that
\begin{eqnarray*}
 \| \nabla f ( x(t))-\nabla f( x^* )\|=\mathcal{O}\left ( \frac{1}{\sqrt{\beta (t)}}\right ), ~\textup{as}~ t\rightarrow+\infty,
\end{eqnarray*}
and
\begin{eqnarray*}
 \| \nabla g ( y(t) )-\nabla g ( y^* ) \|=\mathcal{O}\left ( \frac{1}{\sqrt{\beta (t)}}\right ), ~\textup{as}~ t\rightarrow+\infty.
\end{eqnarray*}
On the other hand,
\begin{eqnarray*}
\begin{split}
&\mathcal{L} ( x(t), y(t),\lambda ^*)-\mathcal{L} ( x^*,y^*,\lambda ^* )\\
=&~f( x(t))-f( x^*)+ \langle A^\top\lambda ^*,x(t)-x^* \rangle\\
&+g ( y(t) )-g( y^*)+ \langle B^\top\lambda ^*,y(t)-y^*\rangle+\frac{1}{2} \| Ax(t)+By(t)-b \|^2\\
\geq&~ \langle \nabla f( x^*),x(t)-x^* \rangle+ \langle A^\top\lambda ^*,x(t)-x^* \rangle\\
&+\langle \nabla g( y^* ),y ( t)-y^*\rangle+\langle B^\top\lambda ^*,y(t)-y^* \rangle+\frac{1}{2} \| Ax(t)+By(t))-b\|^2\\
=&~\frac{1}{2}\| Ax(t)+By(t)-b\|^2,
\end{split}
\end{eqnarray*}
where the second inequality follows from the facts that $f$ and $g$ are convex functions, while the last equality is derived from $-A^\top\lambda^*=\nabla f( x^*)$ and $-B^\top\lambda^*=\nabla g ( y^*)$.
Thus,
\begin{eqnarray}\label{3.9}
\| Ax(t)+By(t)-b \|=\mathcal{O}\left ( \frac{1}{\sqrt{\beta (t)}}\right ), ~\textup{as}~ t\rightarrow+\infty.
\end{eqnarray}
Note that
\begin{eqnarray*}
\begin{split}
&\mathcal{L}  ( x(t), y(t),\lambda ^*)-\mathcal{L} ( x^*,y^*,\lambda ^* )\\
=& ~f( x(t))+g ( y(t) )-f ( x^*)-g( y^* )+\langle\lambda^*,Ax(t)+By(t)-b\rangle+\frac{1}{2}\| Ax(t)+By(t)-b \|^2.
\end{split}
\end{eqnarray*}
So, it is easy to show that
\begin{eqnarray*}
|\varPhi (x(t),y(t))-\varPhi (x^*,y^*)|\leq\mathcal{L}(x(t),y(t),\lambda^*)-\mathcal{L}(x^*,y^*,\lambda^*)+\|\lambda^*\|\|Ax(t)+By(t)-b\|.
\end{eqnarray*}
It follows from (\ref{3.8}) and (\ref{3.9}) that
\begin{eqnarray*}
| \varPhi  ( x(t),y(t))-\varPhi ( x^*,y^*) |=\mathcal{O}\left ( \frac{1}{\sqrt{\beta (t)}}\right ),~\textup{as}~ t\rightarrow+\infty.
\end{eqnarray*}
The proof is complete.\qed
\end{proof}

\begin{remark}
In the special case that $g(y)\equiv 0$ and $B\equiv 0$, the dynamical system (\ref{1.3}) reduces to the dynamical system considered in \cite{HHF2022}. In \cite[Thereom 2.1]{HHF2022}, they show that the convergence rates of the primal-dual gap and the feasibility measure are $\mathcal{O}\left ( \frac{1}{\beta (t)}\right )$ and $\mathcal{O}\left ( \frac{1}{\sqrt{\beta (t)}}\right )$, respectively. Clearly, Theorem \ref{the4.1} extends \cite[Theorem 2.1]{HHF2022} from the linear equality constrained optimization problem (\ref{in1.3}) to the separable convex optimization problem (\ref{1.1}).
\end{remark}

Specially, in the case $\beta(t)=t^{r_1}$ and $\varepsilon(t)=\frac{c}{t^{r_2}}$ with $r_2>r_1>1$, we can improve the convergence rate of primal-dual gap,  objective function error and the feasibility measure.

\begin{theorem}
Let $\beta(t)=t^{r_1}$, $\epsilon (t)=\frac{c}{t^{r_2}}$ with $r_2>r_1>1$ and $c>0$. Suppose that $\frac{1}{\delta}\leq\gamma$ and $r_1\leq\frac{t}{\delta}$. Let $( x,y,\lambda ): [ t_0,+\infty)\rightarrow\mathcal{X}\times\mathcal{Y}\times\mathcal{Z}$ be a solution of the dynamical system \textup{(\ref{1.3})}. Then, it holds that
 \begin{equation*}
 \begin{split}
&\mathcal{L} ( x(t), y(t),\lambda ^* )-\mathcal{L} ( x^*,y^*,\lambda ^* )=o\left ( \frac{1}{t^{r_1-1}}\right ), ~\textup{as}~ t\rightarrow+\infty,\\
&| \varPhi  ( x(t),y(t))-\varPhi ( x^*,y^*) |=o\left ( \frac{1}{\sqrt{t^{r_1-1}}}\right ),~\textup{as}~ t\rightarrow+\infty,\\
&\| Ax(t)+By(t)-b \|=o\left ( \frac{1}{\sqrt{t^{r_1-1}}}\right ), ~\textup{as}~ t\rightarrow+\infty.
\end{split}
\end{equation*}
\end{theorem}
\begin{proof}
Clearly,  we can easily deduce from \eqref{et} that $$\dot{E}(t)\leq\frac{ct^{r_1-r_2}}{2\delta } ( \| x^* \|^2+ \|y^* \|^2 ).$$ Integrating it from $t_0$ to $t$, we have
\begin{equation*}
E(t)\leq E( t_0)+\frac{c\left (  \| x^* \|^2+ \|y^*\|^2\right ) }{2\delta }\int_{t_0}^{t}s^{r_1-r_2}ds,
\end{equation*}
which can be rewritten as
\begin{equation*}
\frac{E(t)}{t}\leq \frac{E( t_0)}{t}+\frac{c\left (  \| x^* \|^2+ \|y^*\|^2\right ) }{2\delta t}\int_{t_0}^{t}  s^{r_1-r_2}ds.
\end{equation*}
Applying Lemma \ref{lem2.1} with $\psi(t)=t$ and $\phi=t^{r_1-r_2-1}$, we obtain $\lim_{t\to+\infty}\frac{E(t)}{t}=0$. Taking $(\ref{3.7})$  into account, it is easy to deduce that
\begin{equation*}
 \mathcal{L} ( x(t), y(t),\lambda ^* )-\mathcal{L} ( x^*,y^*,\lambda ^* )=o\left ( \frac{1}{t^{r_1-1}}\right ), ~\textup{as}~ t\rightarrow+\infty.
 \end{equation*}
This, together with the definition of $\mathcal{L}$, follows that
\begin{eqnarray*}
| \varPhi  ( x(t),y(t))-\varPhi ( x^*,y^*) |=o\left (\frac{1}{t^{r_1-1}}\right ),~\textup{as}~ t\rightarrow+\infty,
\end{eqnarray*}
and
\begin{eqnarray*}
\| Ax(t)+By(t)-b \|=o\left ( \frac{1}{\sqrt{{t^{r_1-1}}}}\right ), ~\textup{as}~ t\rightarrow+\infty.
\end{eqnarray*}
The proof is complete.\qed
\end{proof}

In the sequel, we show that our approach can naturally be extended to the investigation of the separable structured  ``smoooth+nonsmooth'' convex optimization problems. To do this,  we consider the separable structured  convex optimization problem \eqref{1.1} where $f$ is a continuously differentiable convex function  and  $\nabla f$ is Lipschitz continuous, and $g$ is a proper lower semi-continuous convex function. In order to adapt the system \textup{(\ref{1.3})} to this ``smoooth+nonsmooth'' situation, we propose the following differential inclusion system:
\begin{eqnarray}\label{nonsmooth}
\left\{ \begin{array}{ll}
\ddot{x}(t)+\gamma \dot{x}(t)+\beta(t)\left( \nabla f(x(t))+A^\top \lambda(t)+A^\top( Ax(t)+B y(t)-b)+\epsilon(t)x(t)\right)=0,\\
\ddot{y}(t)+\gamma \dot{y}(t)+\beta(t)\left( \partial g(y(t))+B^\top\lambda(t)+B^\top( Ax(t)+B y(t)-b)+\epsilon(t)y(t)\right)\ni0,\\
\dot{\lambda }(t)-\beta(t)\left( A( x(t)+\delta \dot{x}(t))+B(y(t)+\delta\dot{y}(t))-b\right)=0.
\end{array}
 \right.
\end{eqnarray}
Here, $ \partial g$ denotes the subdifferential of $g$.

In the follows, we present the convergence rate results of the differential inclusion system \eqref{nonsmooth}.
\begin{theorem}
Suppose that the assumptions of Theorem $\ref{the4.1}$ are satisfied. Let $( x,y,\lambda ):[ t_0,+\infty)\rightarrow\mathcal{X}\times\mathcal{Y}\times\mathcal{Z}$ is a solution of the  inclusion  system \textup{(\ref{nonsmooth})}. Then, for any $( x^*,y^*,\lambda ^*)\in \Omega$, we have
\begin{equation*}
\begin{split}
&\mathcal{L} ( x(t), y(t),\lambda ^* )-\mathcal{L} ( x^*,y^*,\lambda ^* )=\mathcal{O}\left ( \frac{1}{\beta(t)}\right ), ~\textup{as}~ t\rightarrow+\infty,\\
&|\varPhi  ( x(t),y(t))-\varPhi ( x^*,y^*) |=\mathcal{O}\left ( \frac{1}{\sqrt{\beta(t)}}\right ),~\textup{as}~ t\rightarrow+\infty,\\
&\| Ax(t)+By(t)-b \|=\mathcal{O}\left ( \frac{1}{\sqrt{\beta(t)}}\right ), ~\textup{as}~ t\rightarrow+\infty,\\
&\| \nabla f( x( t))-\nabla f( x^*)\|=\mathcal{O}\left ( \frac{1}{\sqrt{\beta (t)}}\right ), ~\textup{as} ~t\rightarrow+\infty.
\end{split}
\end{equation*}
\end{theorem}
\begin{proof} Let
 $\Phi_\tau(x,y)=f(x)+g_{\tau}(y)$, where $g_\tau$ is the Moreau envelope of $g$ of index $\tau >0$, which is defined as
\begin{equation*}
g_\tau(y)=\mathop{\textup{inf}}\limits_{\mu\in \mathcal{Y}}\left\{g(\mu)+\frac{1}{2\tau}\|y-\mu\|\right\}, \forall y\in \mathcal{Y}.
\end{equation*}
Clearly,  $g_{\tau}(y)$ is continuously differentiable and  $\nabla g_{\tau}$ is Lipschitz continuous.
Therefore, we can study the convergence of the  differential inclusion system \eqref{nonsmooth} by examining the following dynamical system:
\begin{eqnarray}\label{dy}
\small{\left\{ \begin{array}{ll}
\ddot{x}_\tau(t)+\gamma \dot{x}_\tau(t)+\beta(t)\left( \nabla f(x_\tau(t))+A^\top \lambda_\tau(t)+A^\top( Ax_\tau(t)+B y_\tau(t)-b)+\epsilon(t)x_\tau(t)\right)=0,\\
\ddot{y}_\tau(t)+\gamma \dot{y}_\tau(t)+\beta(t)\left( \nabla g_\tau(y_\tau(t))+B^\top\lambda_\tau(t)+B^\top( Ax_\tau(t)+B y_\tau(t)-b)+\epsilon(t)y_\tau(t)\right)=0,\\
\dot{\lambda}_\tau(t)-\beta(t)\left( A( x_\tau(t)+\delta \dot{x}_\tau(t))+B(y_\tau(t)+\delta\dot{y}_\tau(t))-b\right)=0.
\end{array}
 \right.}
\end{eqnarray}

Obviously, the system \eqref{dy} is consistent to the dynamical system \eqref{1.3}. Using a similar argument as in Theorem \ref{the4.1}, we can also show that
\begin{equation}\label{cr}
\begin{split}
&\mathcal{L} ( x_\tau(t), y_\tau(t),\lambda ^* )-\mathcal{L} ( x^*,y^*,\lambda ^* )=\mathcal{O}\left ( \frac{1}{\beta(t)}\right ), ~\textup{as}~ t\rightarrow+\infty,\\
&|\varPhi  ( x_\tau(t),y_\tau(t))-\varPhi ( x^*,y^*) |=\mathcal{O}\left ( \frac{1}{\sqrt{\beta(t)}}\right ),~\textup{as}~ t\rightarrow+\infty,\\
&\| Ax_\tau(t)+By_\tau(t)-b \|=\mathcal{O}\left ( \frac{1}{\sqrt{\beta(t)}}\right ), ~\textup{as}~ t\rightarrow+\infty,\\
&\| \nabla f( x_\tau( t))-\nabla f( x^*)\|=\mathcal{O}\left ( \frac{1}{\sqrt{\beta (t)}}\right ), ~\textup{as} ~t\rightarrow+\infty.\\
\end{split}
\end{equation}
Now, by virtue of the Moreau-Yosida regularization in \cite[Section 2]{LS1997}, there  exists a subsequence $\{(x_\tau(t),y_\tau(t),\lambda_\tau(t))\}_{\tau>0}$ of trajectories, generated by the dynamical system \eqref{dy} , that converges to the trajectories generated by the differential inclusion system \eqref{nonsmooth}. Then, the desired results are obtained by
taking the limit in \eqref{cr} with $\tau\rightarrow0$.
The proof is complete.\qed
\end{proof}

Next, we derive the following proposition, which will be used in the proof of the minimization property and strong convergence of the trajectory.

\begin{proposition}\label{lemma4.1}
Let $( x,y,\lambda ): [ t_0,+\infty)\rightarrow\mathcal{X}\times\mathcal{Y}\times\mathcal{Z}$ be a solution of the dynamical system \textup{(\ref{1.3})} and $ ( x^*,y^*,\lambda ^*)\in\Omega$. Denote
\begin{eqnarray}\label{4.1}
\begin{split}
\tilde{E}(t)=&~ \mathcal{L}( x(t),y(t),\lambda ^*)-\mathcal{L}( x^*,y^*,\lambda ^*)+\frac{\epsilon (t)}{2} \| x(t) \|^2+\frac{\epsilon (t)}{2}\| y(t)\|^2\\
&+\frac{1}{2\beta(t)} \left\| \frac{1}{\delta }( x(t)-x^* )+\dot{x}(t)\right \|^2+\frac{\delta \gamma -1}{2\delta^2\beta (t)}\| x(t)-x^* \|^2\\
&+\frac{1}{2\beta (t)}\left \| \frac{1}{\delta }( y(t)-y^*)+\dot{y}(t) \right\|^2+\frac{\delta \gamma -1}{2\delta ^2\beta (t)} \| y(t)-y^*\|^2+\frac{1}{2\delta\beta (t)} \| \lambda (t)-\lambda ^*\|^2.
\end{split}
\end{eqnarray}
Then,
\begin{eqnarray*}
\begin{split}
\frac{\dot{\beta }(t)}{\beta (t)}\tilde{E}(t)+\dot{\tilde{E}}(t)\leq&  \left( \frac{\dot{\beta }(t)}{\beta(t)}-\frac{1}{\delta }\right )(\mathcal{L}( x(t),y(t),\lambda ^* )-\mathcal{L} ( x^*,y^*,\lambda ^* ))\\
&+\frac{1}{2}\left [\left( \frac{\dot{\beta }(t)}{\beta(t)}-\frac{1}{\delta }\right)\epsilon (t)+\dot{\epsilon }(t)\right ]( \| x(t) \|^2+ \| y(t)\|^2 )\\
&-\frac{\epsilon (t)}{2\delta }( \| x(t)-x^*\|^2+ \| y(t)-y^* \|^2 )\\
&+\frac{\epsilon (t)}{2\delta }( \| x^* \|^2+ \| y^* \|^2 )+\frac{1-\delta \gamma }{\delta\beta (t)}( \| \dot{x}(t)\|^2+ \| \dot{y}(t) \|^2)\\
&-\frac{1}{2\delta } \| Ax(t)+By(t)-b\|^2.
\end{split}
\end{eqnarray*}
\end{proposition}

\begin{proof}
Clearly, $\tilde{E}(t)=\frac{1}{\beta (t)}E(t)$ and
$
\dot{\tilde{E}}(t)=\frac{1}{\beta (t)}\dot{E}(t)-\frac{\dot{\beta }(t)}{\beta (t)^2}E(t).
$
Then,
$
\frac{\dot{\beta }(t)}{\beta (t)}\tilde{E}(t)+\dot{\tilde{E}}(t)=\frac{1}{\beta (t)}\dot{E}(t).
$
From the first inequality of (\ref{et}), we have
\begin{eqnarray*}
\begin{split}
\frac{\dot{\beta }(t)}{\beta (t)}\tilde{E}(t)+\dot{\tilde{E}}(t)\leq&  \left( \frac{\dot{\beta }(t)}{\beta(t)}-\frac{1}{\delta }\right )(\mathcal{L}( x(t),y(t),\lambda ^* )-\mathcal{L} ( x^*,y^*,\lambda ^* ))\\
&+\frac{1}{2}\left [\left( \frac{\dot{\beta }(t)}{\beta(t)}-\frac{1}{\delta }\right)\epsilon (t)+\dot{\epsilon }(t)\right ]( \| x(t) \|^2+ \| y(t)\|^2 )\\
&-\frac{\epsilon (t)}{2\delta }( \| x(t)-x^*\|^2+ \| y(t)-y^* \|^2 )\\
&+\frac{\epsilon (t)}{2\delta }( \| x^* \|^2+ \| y^* \|^2 )+\frac{1-\delta \gamma }{\delta\beta (t)}( \| \dot{x}(t)\|^2+ \| \dot{y}(t) \|^2)\\
&-\frac{1}{2\delta } \| Ax(t)+By(t)-b\|^2.
\end{split}
\end{eqnarray*}
The proof is complete.\qed
\end{proof}

Now, we establish the minimal properties of the trajectories generated by the dynamical system $(\ref{1.3})$.

\begin{theorem}\label{lam4.2}
Suppose that $\epsilon : [ t_0,+\infty )\rightarrow [0,+\infty  )$ and $\beta : [ t_0,+\infty )\rightarrow [0,+\infty )$ with
\begin{equation*}
\int_{t_0}^{+\infty}\epsilon (t)dt< +\infty,~\lim_{t\to+\infty}\beta (t)=+\infty,~\dot{\beta }(t)\leqslant \frac{1}{\delta }\beta (t),\mbox{ and }\frac{1}{\delta}\leq \gamma.
\end{equation*}
Let $( x,y,\lambda ): [ t_0,+\infty)\rightarrow\mathcal{X}\times\mathcal{Y}\times\mathcal{Z}$ be a solution of the dynamical system \textup{(\ref{1.3})}. Then,
\begin{eqnarray*}
\lim_{t\to+\infty}\mathcal{L}( x(t),y(t),\lambda ^*)=\mathcal{L}( x^*,y^*,\lambda ^* ),~~ \forall ~( x^*,y^*,\lambda ^* )\in\Omega.
\end{eqnarray*}
\end{theorem}
\begin{proof}
Note that $\mathcal{L}( x(t),y(t),\lambda ^* )-\mathcal{L} ( x^*,y^*,\lambda ^*)\geq 0$, $\dot{\beta }(t)\leq \frac{1}{\delta }\beta (t)$ and $\frac{1}{\delta }\leq \gamma$. From Proposition \ref{lemma4.1}, we have
\begin{eqnarray*}
\frac{\dot{\beta }(t)}{\beta(t)}\tilde{E}(t)+\dot{\tilde{E}}(t)\leq \frac{\epsilon (t)}{2\delta }(  \| x^*\|^2+\| y^* \|^2 ),~~\forall t\geq t_0 .
\end{eqnarray*}
Then,
\begin{eqnarray*}
\frac{d}{dt} ( \beta (t)\tilde{E}(t) )=\dot{\beta }(t)\tilde{E}(t)+\beta (t)\dot{\tilde{E}}(t)\leq \frac{\beta(t)\epsilon (t)}{2\delta }\left ( \| x^* \|^2+ \| y^* \|^2\right ),~~\forall t\geq t_0 .
\end{eqnarray*}
Integrating it from $t_0$ to $t$, we get
\begin{eqnarray*}
\beta (t)\tilde{E}(t)\leq \beta (t_0)\tilde{E}(t_0)+\frac{\| x^*\|^2+\| y^* \|^2}{2\delta }\int_{t_0}^{t}\beta (s)\epsilon(s)ds.
\end{eqnarray*}
It means that
\begin{eqnarray*}
\tilde{E}(t)\leq \frac{\beta (t_0)\tilde{E}(t_0)}{\beta (t)}+\frac{\| x^* \|^2+ \| y^* \|^2}{2\delta\beta(t) }\int_{t_0}^{t}\beta (s)\epsilon (s)ds.
\end{eqnarray*}
Moreover, from Lemma \ref{lem2.1}, we have
$$
\lim_{t\to+\infty}\frac{1}{\beta (t)}\int_{t_0}^{t}\beta(s)\epsilon (s)ds=0.
$$
This, together with $\lim\limits_{t\to+\infty}\beta (t)=+\infty$ and $\tilde{E}(t)\geq 0$, yields
\begin{eqnarray}\label{3.12}
\lim_{t\to+\infty}\tilde{E}(t)=0.
\end{eqnarray}
By (\ref{4.1}) and (\ref{3.12}), we have
\begin{eqnarray*}
\lim_{t\to+\infty}\mathcal{L}( x(t),y(t),\lambda ^* )=\mathcal{L} ( x^*,y^*,\lambda ^* ), ~~ \forall ~( x^*,y^*,\lambda ^* )\in\Omega.
\end{eqnarray*}
The proof is complete.\qed
\end{proof}

In particular, in the case $\beta(t)=t^{r_1}$ and $\epsilon (t)=\frac{c}{t^{r_2}}$ with $1< r_2\leq r_1+1$ and $c>0$, we show that the convergence rate of the primal-dual gap depends on $r_2$.
\begin{theorem}
Let $\beta(t)=t^{r_1}$, $\epsilon (t)=\frac{c}{t^{r_2}}$ with $1<r_2\leq r_1+1$ and $c>0$. Suppose that $\frac{1}{\delta}\leq\gamma$ and $r_1\leq\frac{t}{\delta}$. Let $( x,y,\lambda ): [ t_0,+\infty)\rightarrow\mathcal{X}\times\mathcal{Y}\times\mathcal{Z}$ be a solution of the dynamical system \textup{(\ref{1.3})}. Then, it holds that
\begin{enumerate}
\item[{\rm (i)}] If $1<r_2<r_1+1$, then $\mathcal{L}(x(t),y(t),\lambda^*)-\mathcal{L}(x^*,y^*,\lambda^*)=\mathcal{O}\left(\frac{1}{t^{r_2-1}}\right),~~\textup{as}~t\rightarrow +\infty$.

\item[{\rm (ii)}] If $r_2=r_1+1$, then $\mathcal{L}(x(t),y(t),\lambda^*)-\mathcal{L}(x^*,y^*,\lambda^*)=\mathcal{O}\left(\frac{\ln t}{t^{r_1}}\right),~~\textup{as}~t\rightarrow +\infty$.

\end{enumerate}
\end{theorem}
\begin{proof}
(i) Since $\beta(t)=t^{r_1}$ and $\epsilon (t)=\frac{c}{t^{r_2}}$, it follows from (\ref{4.1}) that
\begin{eqnarray} \label{3.13}
\begin{split}
\tilde{E}(t)=&\mathcal{L}( x(t),y(t),\lambda ^*)-\mathcal{L}( x^*,y^*,\lambda ^*)+\frac{c}{2t^{r_2}} \|x(t)\|^2+\frac{c}{2t^{r_2}}\|y(t)\|^2\\
&+\frac{1}{2t^{r_1}}\left\| \frac{1}{\delta }( x(t)-x^* )+\dot{x}(t)\right \|^2+\frac{\delta \gamma -1}{2\delta^2t^{r_1}}\| x(t)-x^* \|^2\\
&+\frac{1}{2t^{r_1}}\left \| \frac{1}{\delta }( y(t)-y^*)+\dot{y}(t) \right\|^2+\frac{\delta \gamma -1}{2\delta ^2t^{r_1}} \| y(t)-y^*\|^2+\frac{1}{2\delta t^{r_1}} \| \lambda (t)-\lambda ^*\|^2.
\end{split}
\end{eqnarray}
Since $1<r_2\leq r_1+1$, $r_1\leq\frac{t}{\delta}$ and $\frac{1}{\delta}\leq\gamma$, it is easy to show that all the conditions in Theorem \ref{lam4.2} are satisfied. Using a similar argument as that given for Theorem \ref{lam4.2}, we have
\begin{eqnarray*}
\begin{split}
\tilde{E}(t)\leq& \frac{t^{r_1}_0\tilde{E}(t_0)}{t^{r_1}}+\frac{\| x^* \|^2+ \| y^* \|^2}{2\delta t^{r_1}}\int_{t_0}^{t}\frac{c}{s^{{r_2}-{r_1}}}ds\\
\leq&\frac{t^{r_1}_0\tilde{E}(t_0)}{t^{r_1}}+\frac{c(\| x^* \|^2+ \| y^* \|^2)}{2\delta({r_1}-{r_2}+1)t^{{r_2}-1}}.
\end{split}
\end{eqnarray*}
Then,
\begin{eqnarray*}
t^{{r_2}-1}\tilde{E}(t)\leq\frac{t^{r_1}_0\tilde{E}(t_0)}{t^{{r_1}-{r_2}+1}}+\frac{c(\| x^* \|^2+ \| y^* \|^2)}{2\delta({r_1}-{r_2}+1)}.
\end{eqnarray*}
This, together with (\ref{3.13}) and $1<r_2<r_1+1$, implies
\begin{eqnarray*}
\mathcal{L}(x(t),y(t),\lambda^*)-\mathcal{L}(x^*,y^*,\lambda^*)=\mathcal{O}\left(\frac{1}{t^{r_2-1}}\right),~~\textup{as}~t\rightarrow +\infty.
\end{eqnarray*}

(ii) Since $r_2=r_1+1$, (\ref{3.13}) can be written as:
\begin{eqnarray} \label{3.14}
\begin{split}
\tilde{E}(t)=&\mathcal{L}( x(t),y(t),\lambda ^*)-\mathcal{L}( x^*,y^*,\lambda ^*)+\frac{c}{2t^{r_1+1}} \|x(t)\|^2+\frac{c}{2t^{r_1+1}}\|y(t)\|^2\\
&+\frac{1}{2t^{r_1}}\left\| \frac{1}{\delta }( x(t)-x^* )+\dot{x}(t)\right \|^2+\frac{\delta \gamma -1}{2\delta^2t^{r_1}}\| x(t)-x^* \|^2\\
&+\frac{1}{2t^{r_1}}\left \| \frac{1}{\delta }( y(t)-y^*)+\dot{y}(t) \right\|^2+\frac{\delta \gamma -1}{2\delta ^2t^{r_1}} \| y(t)-y^*\|^2+\frac{1}{2\delta t^{r_1}} \| \lambda (t)-\lambda ^*\|^2.
\end{split}
\end{eqnarray}
Thus,
\begin{eqnarray*}
\begin{split}
\tilde{E}(t)\leq& \frac{t^{r_1}_0\tilde{E}(t_0)}{t^{r_1}}+\frac{\| x^* \|^2+ \| y^* \|^2}{2\delta t^{r_1}}\int_{t_0}^{t}\frac{c}{s}ds\\
=&\frac{t^{r_1}_0\tilde{E}(t_0)}{t^{r_1}}+\frac{c(\| x^* \|^2+ \| y^* \|^2)\ln t}{2\delta t^{r_1}}-\frac{c(\| x^* \|^2+ \| y^* \|^2)\ln t_0}{2\delta t^{r_1}}\\
\leq&\frac{C_4}{2\delta t^{r_1}}+\frac{c(\| x^* \|^2+ \| y^* \|^2)\ln t}{2\delta t^{r_1}},
\end{split}
\end{eqnarray*}
where $C_4\geq 2\delta t^{r_1}_0\tilde{E}(t_0)-c(\| x^* \|^2+ \| y^* \|^2)\ln t_0$. This, together with (\ref{3.14}), implies
\begin{eqnarray*}
\mathcal{L}(x(t),y(t),\lambda^*)-\mathcal{L}(x^*,y^*,\lambda^*)=\mathcal{O}\left(\frac{\ln t}{t^{r_1}}\right),~~\textup{as}~t\rightarrow +\infty.
\end{eqnarray*}
The proof is complete. \qed
\end{proof}

\section{Strong Convergence Results}

In this section, we will show that the trajectory generated by the dynamical system $(\ref{1.3})$ converges strongly to a minimal norm solution of the  separable convex optimization problem $(\ref{1.1})$.

Let $( \bar{x}^*,\bar{y}^*)$ be an element of the minimal norm of the solution set $S$. This means that $( \bar{x}^*,\bar{y}^* )=\mbox{proj}_S0$, where $\mbox{proj}$ is the projection operator. Then, there exists an optimal solution $\bar{\lambda }^*\in \mathcal{Z}$ of problem (\ref{1.6}) such that $(\bar{x}^*,\bar{y}^*,\bar{\lambda }^*)\in\Omega$.
For any $\epsilon(t) > 0$, we define $\mathcal{L}_{\epsilon(t)} :\mathcal{X}\times\mathcal{Y}\rightarrow \mathbb{R}$ by
\begin{eqnarray}\label{5.9}
\mathcal{L}_{\epsilon(t)} ( x,y):=\mathcal{L} ( x,y,\bar{\lambda }^* )+\frac{\epsilon(t) }{2} \| x \|^2+\frac{\epsilon(t) }{2} \| y \|^2.
\end{eqnarray}
Let $\left ( x_{\epsilon(t)} ,y_{\epsilon(t)} \right )$ be the unique solution of the strongly convex minnimization problem
\begin{eqnarray*}
\left ( x_{\epsilon(t)} ,y_{\epsilon(t)} \right )=\mathop{\textup{argmin}}\limits_{x\in\mathcal{X},y\in\mathcal{Y}}{\mathcal{L}_{\epsilon(t)}(x,y)}.
\end{eqnarray*}
The first-order optimality condition gives
\begin{eqnarray}\label{5.1}
\left\{ \begin{array}{ll}
\nabla _x\mathcal{L}_{\epsilon(t)}(x,y) =\nabla _x\mathcal{L}( x_{\epsilon(t)} ,y_{\epsilon(t)} ,\bar{\lambda }^* )+\epsilon(t) x_{\epsilon(t)} =0,\\
\nabla _y\mathcal{L}_{\epsilon(t)}(x,y) =\nabla _y\mathcal{L} ( x_{\epsilon(t)} ,y_{\epsilon(t)} ,\bar{\lambda }^* )+\epsilon(t) y_{\epsilon(t)} =0.
\end{array}
 \right.
\end{eqnarray}
By the classical properties of the Tikhonov regularization given in \cite{bcl2021}, we get
\begin{eqnarray*}
\begin{split}
\| x_{\epsilon(t)} \|\leq \| \bar{x}^* \|,~ \| y_{\epsilon(t)}  \|\leq \| \bar{y}^*\|,~~~\forall \epsilon(t) > 0,
\end{split}
\end{eqnarray*}
and
\begin{eqnarray}\label{5.8}
\lim_{t \to +\infty} \| x_{\epsilon(t)} -\bar{x}^* \|=0,~\lim_{t \to +\infty} \| y_{\epsilon(t)} -\bar{y}^*\|=0.
\end{eqnarray}

The following auxiliary result will play an important role in our study.
\begin{proposition}\label{lem5.1}
Let $( \bar{x}^*,\bar{y}^* )=\mathrm{proj}_S0$ and $( x,y,\lambda ):[t_0,+\infty)\rightarrow\mathcal{X}\times\mathcal{Y}\times\mathcal{Z}$ be a solution of the dynamical system \textup{(\ref{1.3})}. Then,
\begin{eqnarray*}
\begin{split}
\mathcal{L}_{\epsilon (t)}( x(t),y (t) )-\mathcal{L}_{\epsilon(t)}( \bar{x}^*,\bar{y}^* )\geq &\frac{\epsilon (t)}{2} \|(x(t), y(t))-(x_{\epsilon (t)},y_{\epsilon (t)})\|^2\\
&+\frac{\epsilon (t)}{2}( \| x_{\epsilon(t)}\|^2-\| \bar{x}^*\|^2+ \| y_{\epsilon (t)} \|^2- \| \bar{y}^* \|^2).
\end{split}
\end{eqnarray*}
\end{proposition}
\begin{proof}
Clearly, $\mathcal{L}_{\epsilon (t)}$ is $\epsilon (t)$-strongly convex with respect to $x$ and $y$. Then,
\begin{eqnarray*}
\begin{split}
&\mathcal{L}_{\epsilon (t)}( x(t),y(t))-\mathcal{L}_{\epsilon (t)}( x_{\epsilon(t)},y_{\epsilon (t)} )\\
\geq &~\langle \nabla _x\mathcal{L}_{\epsilon (t)} ( x_{\epsilon(t) },y_{\epsilon(t)}),x(t)-x_{\epsilon (t)} \rangle+ \langle \nabla_y\mathcal{L}_{\epsilon (t)}( x_{\epsilon(t) },y_{\epsilon(t) } ),y(t)-y_{\epsilon(t)}\rangle\\
&+\frac{\epsilon (t)}{2}\|(x(t), y(t))-(x_{\epsilon(t)},y_{\epsilon (t)})\|^2.
\end{split}
\end{eqnarray*}
By $(\ref{5.1})$, we have
\begin{eqnarray}\label{prop51}
\mathcal{L}_{\epsilon (t)} ( x(t),y(t))-\mathcal{L}_{\epsilon (t)}( x_{\epsilon(t)},y_{\epsilon(t)} )
\geq\frac{\epsilon(t)}{2} \|(x(t), y(t))-(x_{\epsilon (t)},y_{\epsilon(t)}) \|^2.
\end{eqnarray}
On the other hand, by $\mathcal{L}(x_{\epsilon(t)},y_{\epsilon (t)},\bar{\lambda }^* )-\mathcal{L}( \bar{x}^*,\bar{y}^*,\bar{\lambda }^*)\geq 0$ and (\ref{5.9}), we get
\begin{eqnarray}\label{prop52}
\mathcal{L}_{\epsilon (t)} ( x_{\epsilon (t)},y_{\epsilon (t)} )-\mathcal{L}_{\epsilon(t)}( \bar{x}^*,\bar{y}^*)
\geq \frac{\epsilon (t)}{2}( \| x_{\epsilon (t)} \|^2 -\| \bar{x}^* \|^2 )+\frac{\epsilon (t)}{2}( \| y_{\epsilon (t)} \|^2 -\| \bar{y}^*\|^2).
\end{eqnarray}
Together with $(\ref{prop51})$ and $(\ref{prop52})$, we obtain
\begin{eqnarray*}
\begin{split}
\mathcal{L}_{\epsilon (t)} ( x(t),y (t) )-\mathcal{L}_{\epsilon(t)} ( \bar{x}^*,\bar{y}^*)\geq &\frac{\epsilon(t)}{2} \|(x(t), y(t))-(x_{\epsilon (t)},y_{\epsilon (t)})\|^2\\
&+\frac{\epsilon(t)}{2}( \| x_{\epsilon(t)} \|^2-\| \bar{x}^* \|^2+\| y_{\epsilon (t)} \|^2-\| \bar{y}^* \|^2 ).
\end{split}
\end{eqnarray*}
The proof is complete.\qed
\end{proof}

Now, we establish the following  strong convergence results for the trajectory generated by the dynamical system (\ref{1.3}).
\begin{theorem}\label{TH5.1}
Suppose that
\begin{equation*}
\int_{t_0}^{+\infty}\epsilon ( t )dt<+\infty,~\lim_{t\to+\infty}\beta (t)\epsilon (t)=+\infty,~\dot{\beta }(t)\leqslant \frac{1}{\delta }\beta (t),\mathrm{~and~}\frac{1}{\delta}\leq \gamma.
\end{equation*}
 Let $ ( \bar{x}^*, \bar{y}^*)=\mathrm{proj}_S0$ and $( x,y,\lambda):[t_0,+\infty)\rightarrow\mathcal{X}\times\mathcal{Y}\times\mathcal{Z}$ be a solution of the dynamical system \textup{(\ref{1.3})}. Then,
\begin{eqnarray*}
\liminf_{t\to+\infty}\|  ( x(t),y(t))- ( \bar{x}^*,\bar{y}^*) \|=0.
\end{eqnarray*}
In addition, if there exists $T\geq t_0$ such that the trajectory $\{( x(t),y(t)):t\geq T\}$ stays either in the ball $\mathbb{B}\left( 0,\sqrt{ \| \bar{x}^* \|^2+\| \bar{y}^*\|^2}\right)$ or in its complement, then,
\begin{eqnarray*}
\lim_{t\to+\infty} \| ( x(t),y(t))- ( \bar{x}^*,\bar{y}^* ) \|=0.
\end{eqnarray*}
\end{theorem}

\begin{proof} We analyze the behaviors of the
trajectory $\{( x(t),y(t)):t\geq T\}$ depending on its position with respect to the ball $\mathbb{B}\left ( 0,\sqrt{\| \bar{x}^* \|^2+ \| \bar{y}^* \|^2}\right )$. To do this, we  consider the following three configurations of the trajectory $\{( x(t),y(t)):t\geq T\}$.

$\mathbf{ Case~ I}$: Suppose  that there exists $T\geq t_0$ such that $\{( x(t),y(t)):t\geq T\}$ stays in the complement of  the ball $\mathbb{B}\left ( 0,\sqrt{\| \bar{x}^* \|^2+ \| \bar{y}^* \|^2}\right )$. In other words,
\begin{eqnarray}\label{5.3}
\| x(t)\|^2+\| y(t) \|^2\geq \| \bar{x}^*\|^2+\| \bar{y}^*\|^2,~~\forall t\geq T.
\end{eqnarray}
From $ ( \bar{x}^*,\bar{y}^* )=\mbox{proj}_S0$, there exists $\bar{\lambda }^*\in \mathcal{Z}$ of (\ref{1.6}) such that $( \bar{x}^*,\bar{y}^*,\bar{\lambda }^* )\in\Omega$. Then, we consider the following energy function:
\begin{eqnarray}\label{5.4}
\begin{split}
\hat{E}(t)=&\mathcal{L}_{\epsilon (t)} ( x(t),y(t))-\mathcal{L}_{\epsilon(t)}( \bar{x}^*,\bar{y}^*)+\frac{1}{2\beta(t)}\left \| \frac{1}{\delta}  ( x(t)-\bar{x}^*)+\dot{x}(t) \right\|^2\\
&+\frac{\delta \gamma -1}{2{\delta}^2 \beta (t)} \| x(t)-\bar{x}^* \|^2+\frac{1}{2\beta (t)}\left \| \frac{1}{\delta}  ( y(t)-\bar{y}^*)+\dot{y}(t)\right \|^2\\
&+\frac{\delta \gamma -1}{2{\delta} ^2\beta (t)}\| y(t)-\bar{y}^* \|^2+\frac{1}{2\delta\beta(t) } \| \lambda (t)-\bar{\lambda }^* \|^2.
\end{split}
\end{eqnarray}
Clearly,
$
\hat{E}(t)=\tilde{E}(t)-\frac{\epsilon (t)}{2}( \| \bar{x}^*\|^2+ \| \bar{y}^*\ \|^2 ).
$
Thus,
$$
\dot{\hat{E}}(t)=\dot{\tilde{E}}(t)-\frac{\dot{\epsilon}(t)}{2}\left( \| \bar{x}^*\|^2+\| \bar{y}^* \|^2 \right).
$$
Using a similar argument as that given in Proposition \ref{lemma4.1}, we obtain
\begin{eqnarray*}
\begin{split}
\frac{\dot{\beta }(t)}{\beta (t)}\hat{E}(t)+\dot{\hat{E}}(t)\leq& \left ( \frac{\dot{\beta }(t)}{\beta (t)}-\frac{1}{\delta }\right )(\mathcal{L}( x(t),y(t),\lambda ^*)-\mathcal{L}( x^*,y^*,\lambda ^*))\\
&+\frac{1}{2}\left[\left ( \frac{\dot{\beta }(t)}{\beta(t)}-\frac{1}{\delta }\right)\epsilon(t)+\dot{\epsilon }(t)\right ]\left ( \| x(t) \|^2+\| y(t)\|^2- \| \bar{x}^*\ \|^2-\| \bar{y}^* \|^2\right )\\
&-\frac{\epsilon (t)}{2\delta }\left ( \| x(t)-x^* \|^2+ \| y(t)-y^*\|^2\right )+\frac{1-\delta\gamma}{\delta \beta (t)}\left( \| \dot{x}(t)\|^2+ \| \dot{y}(t)\|^2\right)\\
&-\frac{1}{2\delta }\| Ax(t)+By(t)-b\|^2.
\end{split}
\end{eqnarray*}
Note that $\dot{\beta }(t)\leq \frac{1}{\delta }\beta (t)$, $\frac{1}{\delta }\leq \gamma $ and $\mathcal{L}( x(t),y(t),\lambda ^*)-\mathcal{L}( x^*,y^*,\lambda ^*)\geq0$. Thus, it follows that
\begin{eqnarray*}
\frac{\dot{\beta }(t)}{\beta (t)}\hat{E}(t)+\dot{\hat{E}}(t)\leq \frac{1}{2}\left[\left ( \frac{\dot{\beta }(t)}{\beta (t)}-\frac{1}{\delta }\right)\epsilon(t)+\dot{\epsilon }(t)\right ]\left ( \| x(t)\|^2+\| y(t)\|^2- \| \bar{x}^* \|^2-\| \bar{y}^* \|^2\right ).
\end{eqnarray*}
This, together with (\ref{5.3}), implies
\begin{eqnarray*}
\frac{\dot{\beta }(t)}{\beta(t)}\hat{E}(t)+\dot{\hat{E}}(t)\leq 0,~~\forall t\geq T.
\end{eqnarray*}
It follows that
\begin{eqnarray*}
\frac{d}{dt}\left ( \beta (t)\hat{E}(t)\right )\leq 0,~~\forall t\geq T.
\end{eqnarray*}
Integrating it from $t$ to $T$, we have
\begin{eqnarray*}
\beta (t)\hat{E}(t)\leq \beta (T)\hat{E}(T),~~\forall t\geq T.
\end{eqnarray*}
Using (\ref{5.4}), we have
\begin{eqnarray*}
\mathcal{L}_{\epsilon(t)}( x(t),y(t))-\mathcal{L}_{\epsilon (t)}( \bar{x}^*,\bar{y}^* )\leq \hat{E}(t)\leq\frac{\beta(T)\hat{E}(T)}{\beta (t)},~~\forall t\geq T.
\end{eqnarray*}
By Proposition \ref{lem5.1}, we obtain
\begin{eqnarray*}
\begin{split}
&\frac{\epsilon (t)}{2} \|(x(t), y(t))-(x_{\epsilon(t)},y_{\epsilon (t)}) \|^2+\frac{\epsilon(t)}{2}( \| x_{\epsilon (t)} \|^2- \| \bar{x}^* \|^2+ \| y_{\epsilon (t)} \|^2-\| \bar{y}^* \|^2 )\\
\leq &\frac{\beta (T)\hat{E}(T)}{\beta (t)}.
\end{split}
\end{eqnarray*}
Thus,
\begin{eqnarray*}
0\leq \|(x(t), y(t))-(x_{\epsilon(t)},y_{\epsilon(t)}) \|^2\leq \frac{2\beta (T)\hat{E}(T)}{\beta (t)\epsilon (t)}+\| \bar{x}^* \|^2-\| x_{\epsilon(t)} \|^2+ \| \bar{y}^* \|^2- \| y_{\epsilon(t)}\|^2.
\end{eqnarray*}
Taking  (\ref{5.8}) and $\lim\limits_{t\to+\infty}\beta (t)\epsilon(t)=+\infty$ into account, we get
\begin{eqnarray*}
\lim_{t\to+\infty}\|( x(t),y(t))-( \bar{x}^*,\bar{y}^* ) \|=0.
\end{eqnarray*}

$\mathbf{ Case~ II}$: Assume  that there exists $T\geq t_0$ such that $\{ ( x(t),y(t) ):t\geq T\}$ stays in  the ball $\mathbb{B}\left ( 0,\sqrt{ \| \bar{x}^* \|^2+ \| \bar{y}^*\|^2}\right )$. Equivalently,
$
 \| x(t)\|^2+\| y(t) \|^2<  \| \bar{x}^*\|^2+\| \bar{y}^*\|^2$, $\forall t\geq T.$
This means that
\begin{eqnarray}\label{5.5}
\|(x(t),y(t))\|<\|(\bar{x}^*,\bar{y}^*)\|,~~\forall t\geq T.
\end{eqnarray}
Let $( \bar{x},\bar{y} )$ be a weak cluster point of $\{ ( x(t),y(t))\}_{t\geq t_0}$. This means that there exists a sequence $\{ t_n\}_{n\in \mathbb{N}}\subseteq[T,+\infty)$ such that $t_n\rightarrow +\infty$ and
\begin{eqnarray*}
(x( t_n ),y ( t_n ))\rightharpoonup( \bar{x},\bar{y})~~as~~n\rightarrow +\infty.
\end{eqnarray*}
Since $\mathcal{L}(\cdot,\cdot,\bar{\lambda }^*)$ is weakly lower semicontinuous, we deduce that
\begin{eqnarray}\label{1909}
\mathcal{L}( \bar{x},\bar{y},\bar{\lambda }^*)\leq \liminf_{n\to+\infty}\mathcal{L}( x(t_n),y(t_n),\bar{\lambda }^*).
\end{eqnarray}
By $\lim\limits_{t\to+\infty}\epsilon (t)=0$ and $\lim\limits_{t\to+\infty}\beta (t)\epsilon (t)=+\infty$, we have $ \lim\limits_{t\to+\infty}\beta(t)=+\infty$.
Then, it follows from Theorem \ref{lam4.2} that
\begin{eqnarray}\label{19009}
\lim_{n\to+\infty}\mathcal{L} ( x(t_n),y(t_n),\bar{\lambda} ^*)=\mathcal{L}( \bar{x}^*,\bar{y}^*,\bar{\lambda} ^*).
\end{eqnarray}
Thus, we deduce from  $(\ref{1909})$ and  $(\ref{19009})$ that
$$
\mathcal{L}( \bar{x},\bar{y},\bar{\lambda }^* )\leq \mathcal{L} ( \bar{x}^*,\bar{y}^*,\bar{\lambda }^* ).
$$
Moreover, it follows from $( \bar{x}^*,\bar{y}^*,\bar{\lambda }^* )\in\Omega$ that
\begin{eqnarray*}
\mathcal{L}( \bar{x}^*,\bar{y}^*,\bar{\lambda }^*)=\min_{x\in \mathcal{X},y\in \mathcal{Y}}{\mathcal{L} ( x,y,\bar{\lambda }^* )}\leq \mathcal{L}( \bar{x},\bar{y},\bar{\lambda }^*),
\end{eqnarray*}
where the inequality holds due to $(\ref{5555})$.
As a consequence,
\begin{eqnarray*}
\mathcal{L} ( \bar{x},\bar{y},\bar{\lambda }^*)= \mathcal{L}( \bar{x}^*,\bar{y}^*,\bar{\lambda }^*)=\min_{x\in \mathcal{X},y\in \mathcal{Y}}{\mathcal{L} ( x,y,\bar{\lambda }^*)}.
\end{eqnarray*}
This implies
$$
 ( \bar{x},\bar{y} )\in \mathop{\textup{argmin}}\limits_{x\in\mathcal{X},y\in\mathcal{Y}}{\mathcal{L}(x,y,\bar{\lambda}^*)}.
$$
Thus, together with $(\ref{1.4})$, we have $ ( \bar{x},\bar{y} )\in S$. On the other hand, from $(\ref{5.5})$ and the lower semi-continuity in the weak topology of $ \|\cdot\|$, we deduce that
\begin{eqnarray*}
\| (\bar{x},\bar{ y})\|\leq\liminf_{n\to+\infty}{ \|  ( x ( t_n) , y( t_n) )\|}\leq  \| (\bar{x}^*,\bar{ y}^*)\|.
\end{eqnarray*}
Combining this with  $ ( \bar{x},\bar{y})\in S$, $ ( \bar{x}^*,\bar{y}^* )=\mbox{proj}_S0$ and $S$ being a convex set, we obtain
\begin{eqnarray}\label{lower2}
( \bar{x},\bar{y})= ( \bar{x}^*,\bar{y}^*).
\end{eqnarray}
Thus, $\{ ( x(t),y(t) )\}_{t\geq t_0}$ has a unique cluster point $( \bar{x}^*,\bar{y}^*)$. Therefore,
\begin{eqnarray*}
( x(t),y(t))\rightharpoonup ( \bar{x}^*,\bar{y}^* )~~as~~t\rightarrow +\infty.
\end{eqnarray*}
In order to obtain the strong convergence, we use (\ref{5.5}), (\ref{lower2}) and the weak topology of $ \|\cdot\|$ to show that
\begin{eqnarray*}
\| ( \bar{x}^*,\bar{y}^* ) \|\leq \liminf_{t\to+\infty} \|( x(t),y(t))\|\leq\limsup_{t\to+\infty}{ \|  ( x(t) , y(t)) \|}\leq \| (\bar{x}^*,\bar{ y}^*)\|.
\end{eqnarray*}
This means that
$$
\lim_{t\to+\infty} \| ( x(t),y(t))\|= \| ( \bar{x}^*,\bar{y}^*)\|.
$$
Thus,
\begin{eqnarray*}
\lim_{t\to+\infty} \| ( x(t),y(t) )-( \bar{x}^*,\bar{y}^* )\|=0.
\end{eqnarray*}

$\mathbf{ Case~ III}$: For any $T\geq t_0$, there exists $ t\geq T$  such that
$\|(\bar{x}^*,\bar{y}^*)\|>\|(x(t),y(t))\|,$
and
also there exists  $ s\geq T$  such that
$\|(\bar{x}^*,\bar{y}^*)\|<\|(x(s),y(s))\|.$ Equivalently, the trajectory $\{ (x(t),y(t)):t\geq T\}$ remains neither in the ball $\mathbb{B}\left ( 0,\sqrt{\| \bar{x}^*\|^2+\| \bar{y}^*\|^2}\right)$, nor in the complement of $\mathbb{B}\left ( 0,\sqrt{\| \bar{x}^*\|^2+ \| \bar{y}^* \|^2}\right)$. Thus, there exists a sequence $ \left\{ t_n \right\}_{n\in\mathbb{ N}}\subseteq [ t_0,+\infty )$ such that $t_n\rightarrow +\infty$ as $n\rightarrow +\infty$, and
\begin{equation*}
\left \| x\left ( t_n\right )\right \|^2+\left \| y\left ( t_n\right )\right \|^2=\left \| \bar{x}^*\right \|^2+\left \| \bar{y}^*\right \|^2, \forall n\in \mathbb{N}.
\end{equation*}
This means that
\begin{eqnarray}\label{weakpoint}
\|(x(t_n),y(t_n))\|=\|(\bar{x}^*,\bar{y}^*)\|,~~\forall n\in \mathbb{N}.
\end{eqnarray}
Let $\left ( \hat{x},\hat{y}\right )$ be a weak sequential cluster point of $\left \{ (x\left ( t_n\right ),y\left ( t_n\right ))\right\}_{n\in \mathbb{N}}$. Using a similar argument as that given for Case II, we deduce that $\left ( \hat{x},\hat{y}\right )=\left ( \bar{x}^*,\bar{y}^*\right )$ and
\begin{eqnarray*}
\left ( x\left ( t_n\right ),y\left ( t_n\right )\right )\rightharpoonup\left ( \bar{x}^*,\bar{y}^*\right ).
\end{eqnarray*}
This, together with $(\ref{weakpoint})$, implies
$
\lim_{n\to+\infty}\| ( x( t_n),y( t_n))-( \bar{x}^*, \bar{y}^* )\|=0.
$
Thus,
\begin{eqnarray*}
\liminf_{t\to+\infty} \|  ( x(t),y(t))- ( \bar{x}^*, \bar{y}^* )\|=0.
\end{eqnarray*}
The proof is complete.\qed
\end{proof}
\begin{remark}
In \cite[Theorem 4.2]{ZHF2024}, Zhu et al. have shown that  the trajectory generated by its dynamical system converges strongly to the minimal solution of the linear equality constrained convex optimization problem (\ref{in1.3}) when the Tikhonov regularization parameter $\epsilon \left ( t\right )$ decreases slowly to zero. Using a similar argument as that given in \cite[Theorem 4.2]{ZHF2024}, we show that the trajectory $\left(x\left(t\right),y\left(t\right)\right)$ generated by the dynamical system (\ref{1.3}) converges strongly to a minimal solution of the separable convex optimization problem with linear equality constraints. Clearly, Theorem \ref{TH5.1} covers \cite[Theorem 4.2]{ZHF2024} as a special case.
\end{remark}

\section{Numerical Experiments}

In this section, motivated by the examples reported in \cite[Section 6.4]{acfr2022}, \cite[Section 5]{zdx2024} and \cite[Section 4]{L2023}, we present some examples to illustrate our theoretical convergence results of the dynamical systems (\ref{1.3}) and \eqref{nonsmooth}. In our numerical experiments, all codes are run on a PC (with 1.600GHz Dual-Core Intel Core i5 and 8GB memory). The dynamical systems are solved by ode45 adaptive method in MATLAB R2018a.
\begin{example} Let $x:=(x_1,x_2,x_3)\in \mathbb{R}^3$ and $y\in\mathbb{R}$.
Consider the following linear equality constrained optimization problem:
\begin{eqnarray}\label{ex1}
\left\{ \begin{array}{ll}
&\mathop{\mbox{min}}\limits_{x\in\mathbb{R}^3,y\in\mathbb{R}}~~{\varPhi  ( x,y ):={( mx_1+nx_2+ex_3)}^{2}+{dy}^2}\\
&~~~\mbox{s.t.}~~~~~Ax+By=b.
\end{array}
\right.
\end{eqnarray}
\end{example}
where $\varPhi:\mathbb{R}^3\times \mathbb{R}\rightarrow \mathbb{R}$, $A:=(m,-n,e)$, $B:=(d)$ with $m,n,e,d\in \mathbb{R}\backslash\{0\}$, and $b:=0$.

It is easy to see that the optimal solution set of (\ref{ex1}) is $S=\left\{(x_1,0,-\frac{m}{e}x_1,0):x_1\in \mathbb{R}\right\}$ and the optimal value is 0. Moreover, the minimal norm solution of (\ref{ex1}) is $(\bar{x}^*,\bar{y}^*)=(0,0,0,0)^\top$. In the following numerical experiments, we take the initial condition $x( 1)= ( 1,1,1)^\top$, $\dot{x}( 1)= ( 1,1,1 )^\top$, $y( 1 )=1$, $\dot{y}(1)=1$ and $\lambda ( 1 )=1$.

In the first experiment, the dynamical system (\ref{1.3}) is  solved on the time interval $[1,100]$. We take $m=5$, $n=1$, $e=1$, $d=5$, $\gamma=0.25$, $\delta=9$, $\beta(t)=t^{\frac{1}{2}}$ and $\epsilon ( t)=\frac{3}{t^{r}}$. For any $ (x^*, y^*)\in S$ and $(\bar{x}^*,\bar{y}^*)=(0,0,0,0)^\top$, we investigate the evolution of the iterate error $\| ( x ( t ),y( t) )- ( \bar{x}^*,\bar{y}^* ) \|$ and the energy error $|\varPhi( x( t ),y ( t ) )-\varPhi  ( x^*,y^* )|$ with different values of $r$. The results are depicted in Figure \ref{fig1}.
\begin{figure}[h]
\centering
\begin{subfigure}[t]{0.45\textwidth}
\rotatebox{90}{\scriptsize{~~~~~~~~\small{$\left \| \left ( x(t),y(t)\right)- \left ( \bar{x}^*,\bar{y}^*\right )\right \|$}}}
\hspace{-1mm}
\includegraphics[width=1\linewidth]{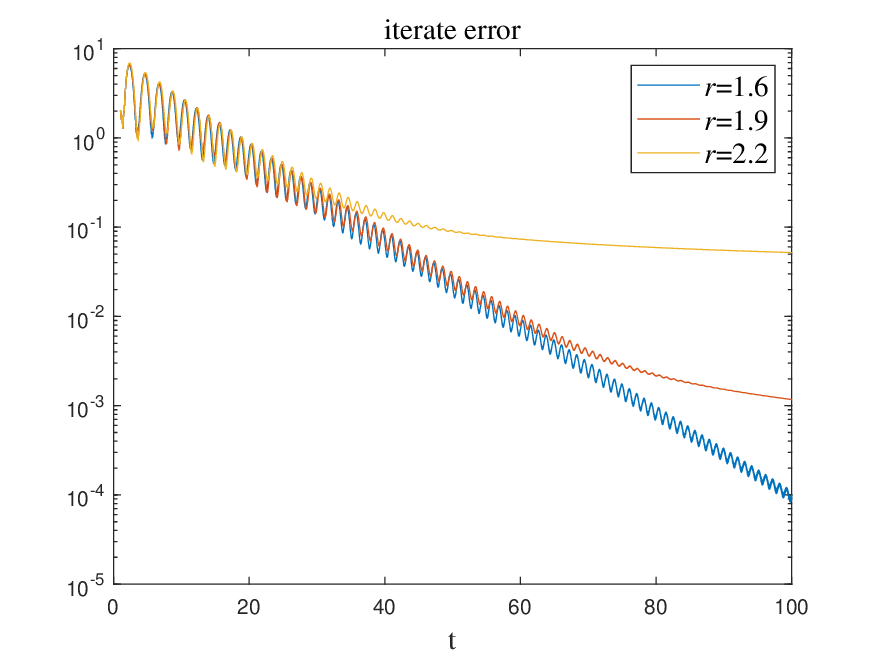}
\end{subfigure}
\hfill
\begin{subfigure}[t]{0.45\textwidth}
\rotatebox{90}{\scriptsize{~~~~~~~\small{$|\varPhi \left ( x(t),y(t)\right )-\varPhi \left ( x^*,y^*\right )|$}}}
\hspace{-1mm}
\includegraphics[width=1\linewidth]{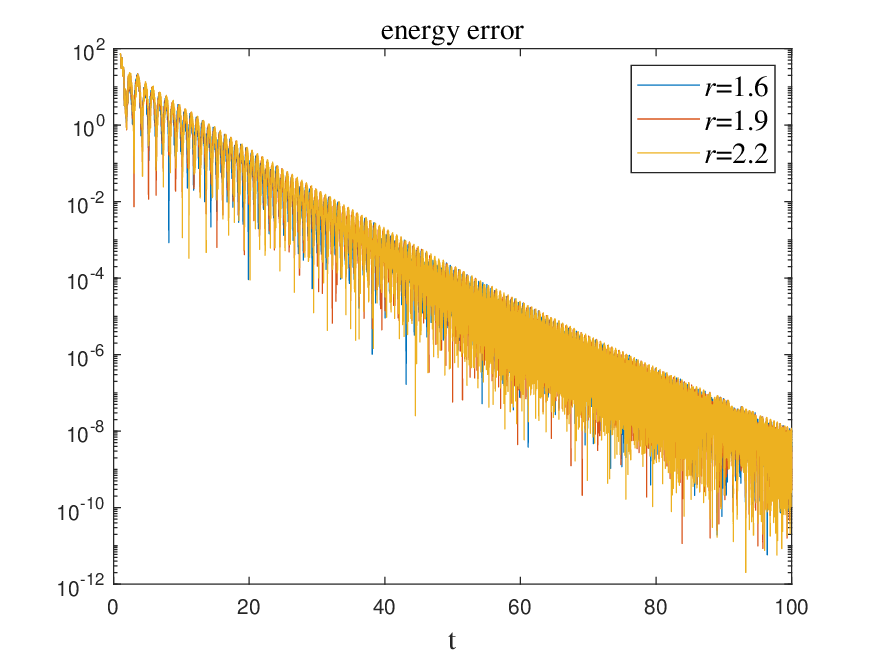}
\end{subfigure}
\caption{Error analysis with different Tikhonov regularization coefficients in the dynamical system (\ref{1.3})}
\label{fig1}
\end{figure}

As shown in Figure \ref{fig1}, the best convergence result for the iterate error $ \| ( x ( t ),y ( t ) )-   ( \bar{x}^*,\bar{y}^* )\|$ is achieved for $r=1.6$.  Furthermore, the energy error $|\varPhi ( x( t),y( t ) )-\varPhi( x^*,y^*)|$ is not sensitive to the changes of the Tikhonov regularization coefficient.

In the second numerical experiment, the dynamical system (\ref{1.3}) is  solved on the time interval $[1,30]$. By choosing different values of $m$, $n$, $e$ and $d$, we reveal the influence of the Tikhonov regularization function $\epsilon(t)$ on the strong convergence of the primal trajectory $\left\{(x(t),y(t))\right\}_{t\geq t_0}$ generated by the dynamical system (\ref{1.3}). To this end, we consider the following cases:

$\mathbf{ Case~ I}$: $\epsilon(t) \neq0$. In this case, let $\gamma=10$, $\beta(t)=t^{\frac{1}{2}}$, $\delta=0.5$ and $\epsilon(t)=\frac{15}{t^{1.6}}$ in the dynamical system (\ref{1.3}). The results are depicted in Figure \ref{fig2}.

\begin{figure}[h]
\centering
\begin{subfigure}[t]{0.45\textwidth}
\rotatebox{90}{\scriptsize{~~~~~~~~~~~~~~~~~~~~$(x(t),y(t))$}}
\hspace{-1mm}
\includegraphics[width=1\linewidth]{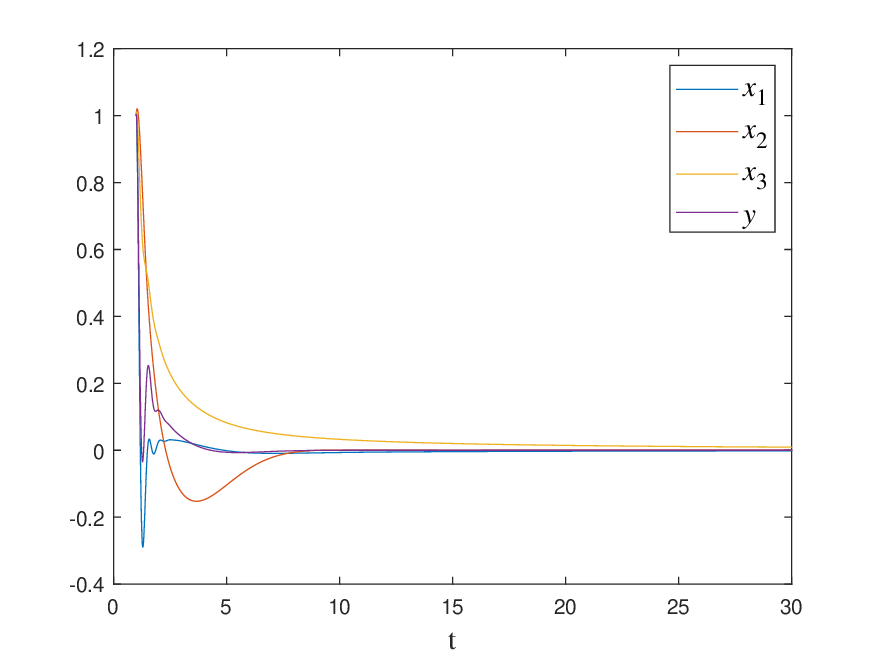}
\end{subfigure}
\hfill
\begin{subfigure}[t]{0.45\textwidth}
\rotatebox{90}{\scriptsize{~~~~~~~~~~~~~~~~~~~~$(x(t),y(t))$}}
\hspace{-1mm}
\includegraphics[width=1\linewidth]{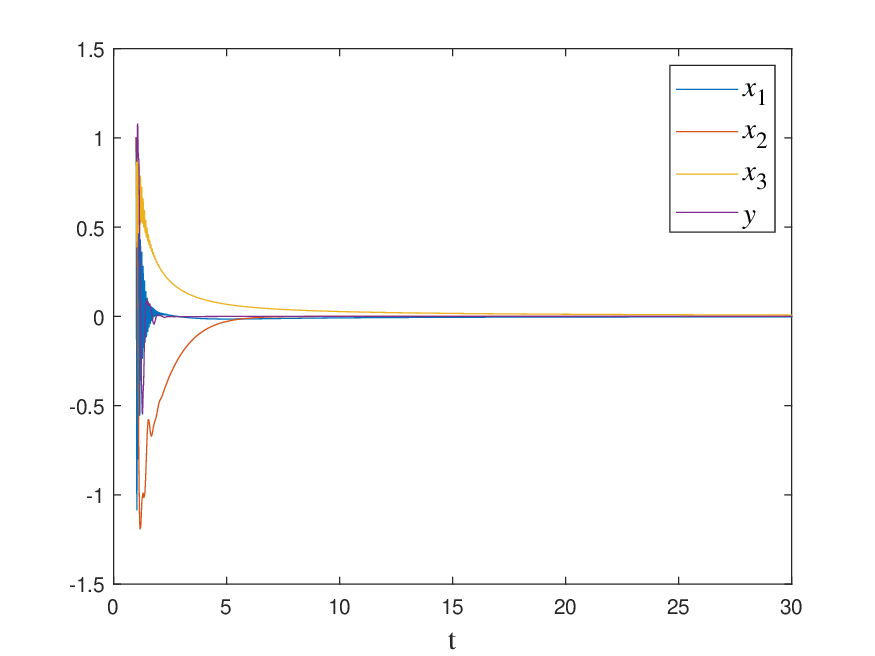}
\end{subfigure}
\begin{minipage}[t]{0.45\textwidth}
        \centering
        \vspace{-3mm}
        \scriptsize{$m=5; n=1; e=1; d=5$}
    \end{minipage}
    \hfill
    \begin{minipage}[t]{0.45\textwidth}
        \centering
        \vspace{-3mm}
        \scriptsize{$m=50; n=10; e=15;d=10$}
    \end{minipage}
\caption{The behaviors of the trajectory $\left\{(x(t),y(t))\right\}_{t\geq t_0}$ generated by the dynamical system (\ref{1.3}) with $\epsilon(t) \neq0$.}
\label{fig2}
\end{figure}
Clearly, under different choices of $m$, $n$, $e$ and $d$, the primal trajectory $\left\{(x(t),y(t))\right\}_{t\geq t_0}$ generated by the dynamical system (\ref{1.3})  with $\epsilon(t) \neq0$ (i.e., with Tikhonov regularization) converges strongly to the minimal norm solution $(\bar{x}^*,\bar{y}^*)=(0,0,0,0)^\top$.

$\mathbf{ Case~{II}}$: $\epsilon (t) \equiv0$. In this case, let $\gamma=10$, $\beta(t)=t^{\frac{1}{2}}$, $\delta=0.5$ and $\epsilon(t)\equiv 0$ in the dynamical system (\ref{1.3}). The results are depicted in Figure \ref{fig3}.
\begin{figure}[h]
\centering
\begin{subfigure}[t]{0.45\textwidth}
\rotatebox{90}{\scriptsize{~~~~~~~~~~~~~~~~~~~~$(x(t),y(t))$}}
\hspace{-1mm}
\includegraphics[width=1\linewidth]{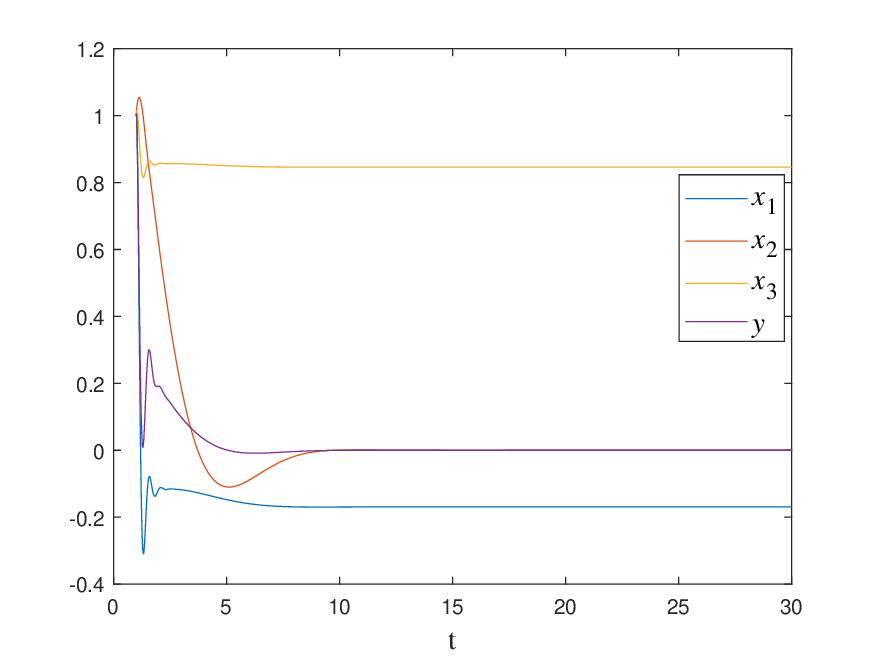}
\end{subfigure}
\hfill
\begin{subfigure}[t]{0.45\textwidth}
\rotatebox{90}{\scriptsize{~~~~~~~~~~~~~~~~~~~~$(x(t),y(t))$}}
\hspace{-1mm}
\includegraphics[width=1\linewidth]{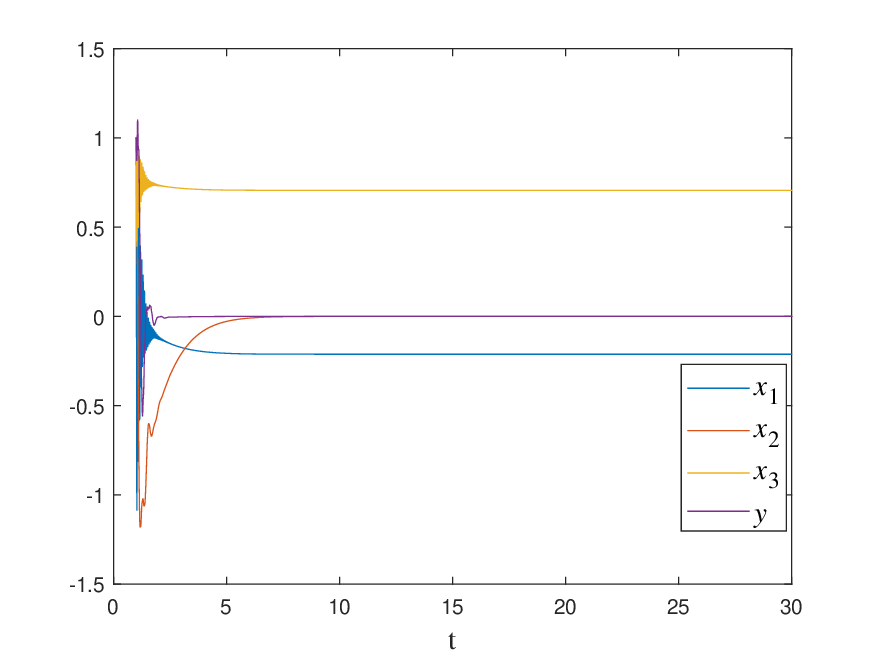}
\end{subfigure}
\begin{minipage}[t]{0.45\textwidth}
        \centering
        \vspace{-3mm}
        \scriptsize{$m=5; n=1; e=1; d=5$}
    \end{minipage}
    \hfill
    \begin{minipage}[t]{0.45\textwidth}
        \centering
        \vspace{-3mm}
        \scriptsize{$m=50; n=10; e=15; d=10$}
    \end{minipage}
\caption{The behaviors of the trajectory $\left\{(x(t),y(t))\right\}_{t\geq t_0}$ generated by the dynamical system (\ref{1.3}) with $\epsilon (t) \equiv0$.}
\label{fig3}
\end{figure}
\vspace{0cm}

Clearly, under different choices of $m$, $n$, $e$ and $d$, the trajectory $\left\{(x(t),y(t))\right\}_{t\geq t_0}$ generated by the dynamical (\ref{1.3}) with $\epsilon (t) \equiv0$ (i.e., without Tikhonov regularization) can not converge to the minimum norm solution $(\bar{x}^*,\bar{y}^*)=(0,0,0,0)^\top$.

In a word, Figures \ref{fig2} and  \ref{fig3}  show  that
Tikhonov regularization term can guarantee that the generated trajectory $\left\{(x(t),y(t))\right\}_{t\geq t_0}$ generated by the dynamical system (\ref{1.3}) converges strongly to the minimal norm solution.

\begin{example}\cite{acfr2022} Let $x:=(x_1,x_2)\in \mathbb{R}^2$ and $y:=(y_1,y_2)\in\mathbb{R}^2$. Consider the following strongly convex quadratic programming problem:
\begin{eqnarray}\label{ex2}
\left\{ \begin{array}{ll}
&\mathop{\mbox{min}}\limits_{x\in\mathbb{R}^2,y\in\mathbb{R}^2}~~{\varPhi  ( x,y ):={\|x-(1,1)^\top\|+\|y\|^2}}\\
&~~~~\mbox{s.t.}~~~~~~x-y-(x_2,0)^\top=(0,0)^\top.
\end{array}
\right.
\end{eqnarray}
\end{example}

In the following two numerical experiments, we compare the dynamical system (\ref{1.3}) with the dynamical systems (\ref{in1.5}) and (\ref{in1.4}).
Here,  the dynamical systems (\ref{in1.5}), (\ref{in1.4}) and (\ref{1.3}) are solved  on the time interval
$[1,100]$. We take the initial condition $x(1)=(1,1)^\top$, $y(1)=(1,1)^\top$, $\lambda(1)=(1,1)^\top$, $\dot{x}(1)=(1,1)^\top$, $\dot{y}(1)=(1,1)^\top$ and $\dot{\lambda}(1)=(1,1)^\top$. By using the MATLAB function quadprog, the optimal value of (\ref{ex2}) is obtained as 0.6. We test the dynamical systems (\ref{in1.5}), (\ref{in1.4}) and (\ref{1.3}) under the following parameters setting:

$\mathbf{ Case~ I}$: The dynamical systems (\ref{in1.5}) and (\ref{in1.4}) incorporate constant damping coefficients.

$ \blacktriangleright$ Dynamical system (\ref{in1.5}): ~~$\gamma(t)\equiv7$ and $\delta(t)=0.2$;

$ \blacktriangleright$ Dynamical system (\ref{in1.4}): ~~$\gamma(t)\equiv7$, $\beta(t)=t^{0.1}$, $\mu=1$ and $a(t)=0.2$;

$ \blacktriangleright$ Dynamical system (\ref{1.3}):~~ $\gamma=9$, $\beta(t)=t^{0.1}$, $\epsilon(t)=\frac{1}{t^2}$ and $\delta=0.2$.

The behaviors of the energy error and the feasibility measure of the dynamical systems (\ref{in1.5}), (\ref{in1.4}), and (\ref{1.3}) with $r\in\{0,0.1,0.4\}$ are depicted in Figure \ref{fig4}.

\begin{figure}[h]
\centering
\setlength{\abovecaptionskip}{0.cm}
\begin{subfigure}[t]{0.45\textwidth}
\rotatebox{90}{\scriptsize{~~~~~~~~$|\varPhi \left ( x(t),y(t)\right )-\varPhi \left ( x^*,y^*\right )|$}}
\hspace{-1mm}
\includegraphics[width=1\linewidth]{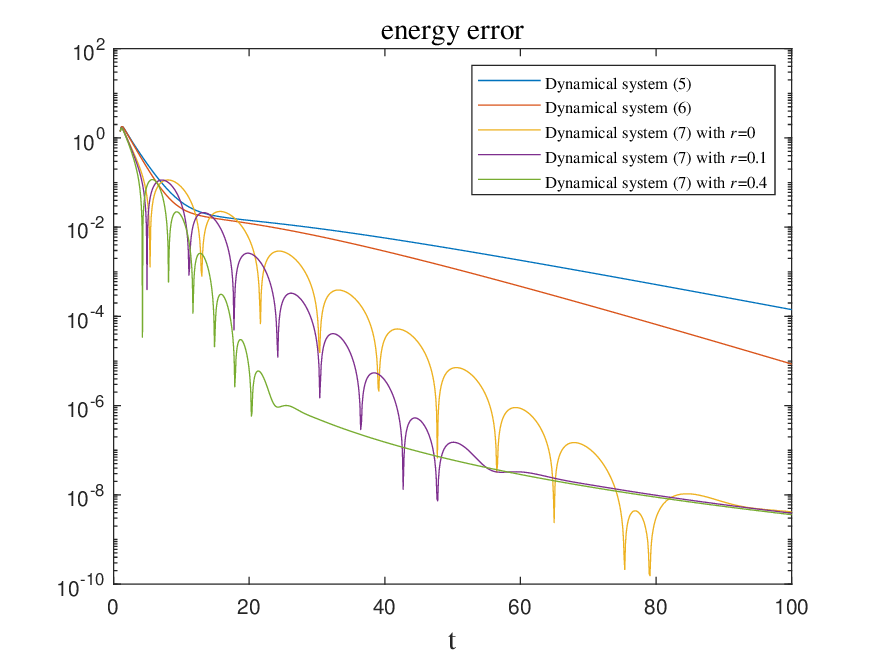}
\end{subfigure}
\hfill
\begin{subfigure}[t]{0.45\textwidth}
\rotatebox{90}{\scriptsize{~~~~~~~~~~~~$\|Ax(t)+By(t)-b\|$}}
\hspace{-1mm}
\includegraphics[width=1\linewidth]{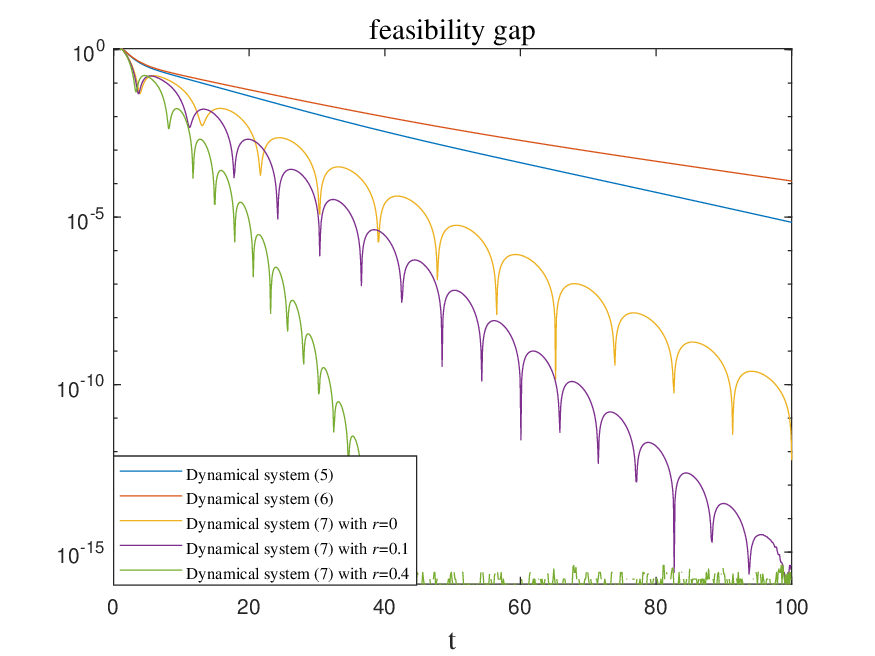}
\end{subfigure}
\begin{minipage}[t]{0.45\textwidth}
        \centering
        \vspace{-3mm}
    \end{minipage}
    \hfill
    \begin{minipage}[t]{0.45\textwidth}
        \centering
        \vspace{-3mm}
    \end{minipage}
\caption{The energy error and the feasibility measure of the dynamical systems (\ref{in1.5}), (\ref{in1.4}) and (\ref{1.3}). }
\label{fig4}
\end{figure}
As shown in Figure \ref{fig4}, the yellow, purple and green curves show that choosing a faster-growing time scaling parameter $\beta(t)$ can produce better convergence rates. Furthermore, from all these curves, it is easy to see that the dynamical system \eqref{1.3} preforms better than the dynamical systems \eqref{in1.5} and \eqref{in1.4}.

{$\mathbf{ Case~ II}$: The dynamical systems (\ref{in1.5}) and (\ref{in1.4}) incorporate asymptotically vanishing damping coefficients.

$ \blacktriangleright$ Dynamical system (\ref{in1.5}): ~~$\gamma(t)=\frac{6}{t}$ and $\delta(t)=\frac{t}{4}$;

$ \blacktriangleright$ Dynamical system (\ref{in1.4}): ~~$\gamma(t)=\frac{6}{t}$, $\beta(t)=t$, $\mu=1$ and $a(t)=\frac{t}{4}$;

$ \blacktriangleright$ Dynamical system (\ref{1.3}):~~ $\gamma=8$, $\beta(t)=t$, $\epsilon(t)=\frac{1}{t^3}$ and $\delta=0.8$.

The behaviors of the energy error and the feasibility measure of the dynamical systems (\ref{in1.5}), (\ref{in1.4}), and (\ref{1.3}) are depicted in Figure \ref{fig5}.
\begin{figure}[h]
\centering
\setlength{\abovecaptionskip}{0.cm}
\begin{subfigure}[t]{0.45\textwidth}
\rotatebox{90}{\scriptsize{~~~~~~~~$|\varPhi \left ( x(t),y(t)\right )-\varPhi \left ( x^*,y^*\right )|$}}
\hspace{-1mm}
\includegraphics[width=1\linewidth]{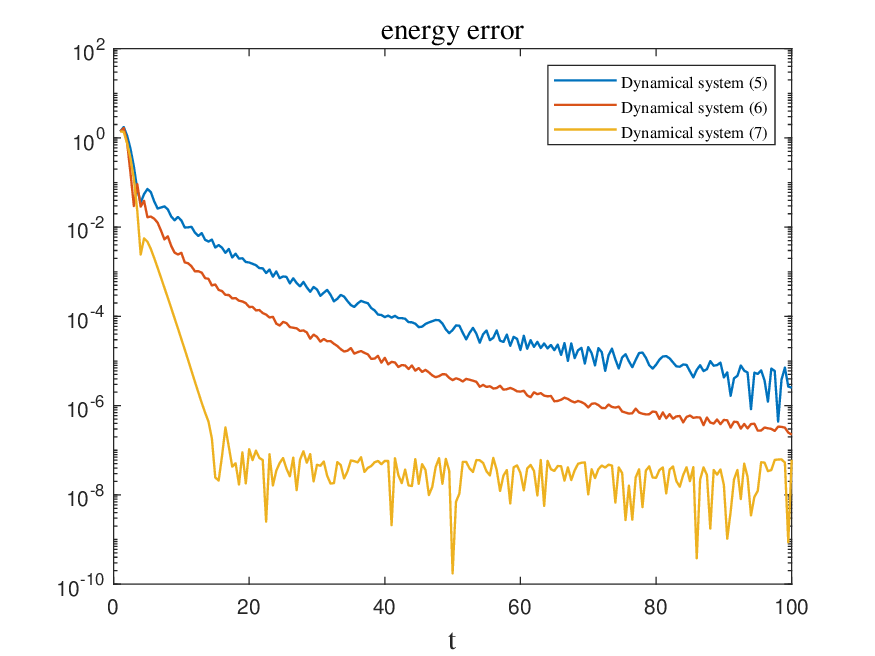}
\end{subfigure}
\hfill
\begin{subfigure}[t]{0.45\textwidth}
\rotatebox{90}{\scriptsize{~~~~~~~~~~~~$\|Ax(t)+By(t)-b\|$}}
\hspace{-1mm}
\includegraphics[width=1\linewidth]{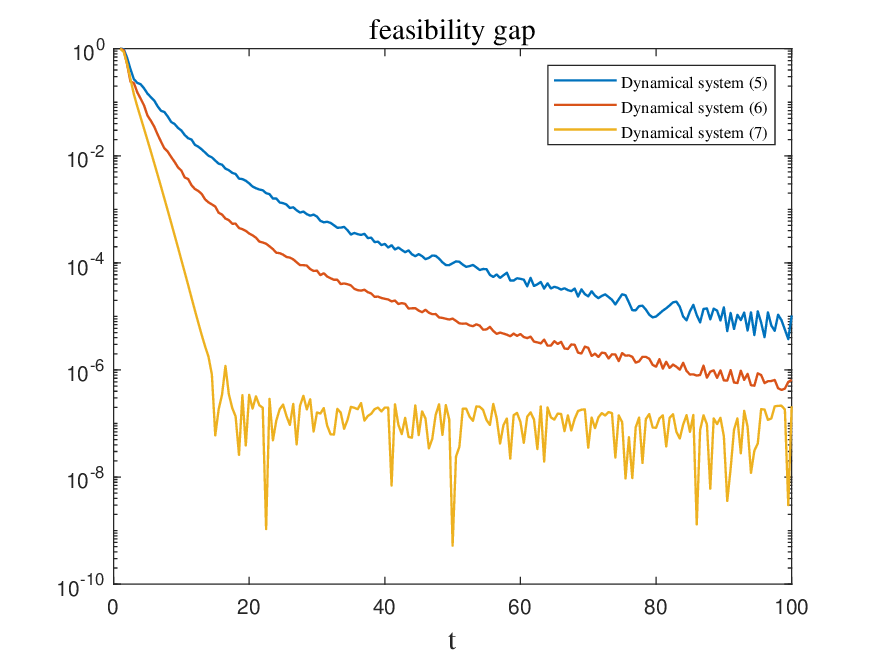}
\end{subfigure}
\begin{minipage}[t]{0.45\textwidth}
        \centering
        \vspace{-3mm}
    \end{minipage}
    \hfill
    \begin{minipage}[t]{0.45\textwidth}
        \centering
        \vspace{-3mm}
    \end{minipage}
\caption{The energy error and the feasibility measure of the dynamical systems (\ref{in1.5}), (\ref{in1.4}) and (\ref{1.3}). }
\label{fig5}
\end{figure}

As shown in Figure \ref{fig5}, when the dynamical systems  (\ref{in1.5}) and (\ref{in1.4}) incorporate asymptotically vanishing damping,  the dynamical system (\ref{1.3}) also performs better  than the dynamical systems (\ref{in1.5}) and (\ref{in1.4}).

At the end of this section,  we give the following example to illustrate the theoretical convergence results of  the differential inclusion system \eqref{nonsmooth}.

\begin{example} Let $x\in \mathbb{R}^n$ and $y\in \mathbb{R}^n$. Consider the following $l_1$-$l_2$ minimization problem
\begin{eqnarray*}
\left\{ \begin{array}{ll}
&\mathop{\mbox{min}}\limits_{x\in\mathbb{R}^n,y\in\mathbb{R}^n}~~{\varPhi  ( x,y ):={\frac{\iota}{2}\|x\|_2^2+\|y\|_1}}\\
&~~~~\mbox{s.t.}~~~~~~Ax+By=b,
\end{array}
\right.
\end{eqnarray*}
where $A\in\mathbb{R}^{m\times n}$, $B\in\mathbb{R}^{m\times n}$ and $b\in\mathbb{R}^{m}$.
\end{example}}
In the following numerical experiment,  we test the influence of the time scaling coefficient on dynamical system \eqref{nonsmooth} under different dimensions.

The dynamical system \eqref{nonsmooth} is solved on  the time interval $[0.1, 5]$. Let  $A$ and $B$ be generated by standard Gaussian distribution. Moreover, let $b=Ax^*+By^*$, where $x^*\in \mathbb{R}^n$ and $y^*\in \mathbb{R}^n$ are generated by the Gaussian distribution $\mathcal{N}(0, 4)$ in $[-2, 2]$ with only $10\%$ non-zero elements.

Take $\iota=0.1$, $\gamma=5$, $\delta=0.5$ and $\varepsilon(t)=\frac{1}{t^2}$. The behaviors of the residual error $\|Ax(t)+By(t)-b\|$ and the relative error$\frac{\|(x(t),y(t))-(x^*,y^*)\|}{\|(x^*,y^*)\|}$ are depict in Figures \ref{fig6} and   \ref{fig7}.
\begin{figure}[h]
\centering
\setlength{\abovecaptionskip}{0.cm}
\begin{subfigure}[t]{0.45\textwidth}
\rotatebox{90}{\scriptsize{~~~~~~~~~~~$\frac{\|(x(t),y(t))-(x^*,y^*)\|}{\|(x^*,y^*)\|}$  }}
\hspace{-1mm}
\includegraphics[width=1\linewidth]{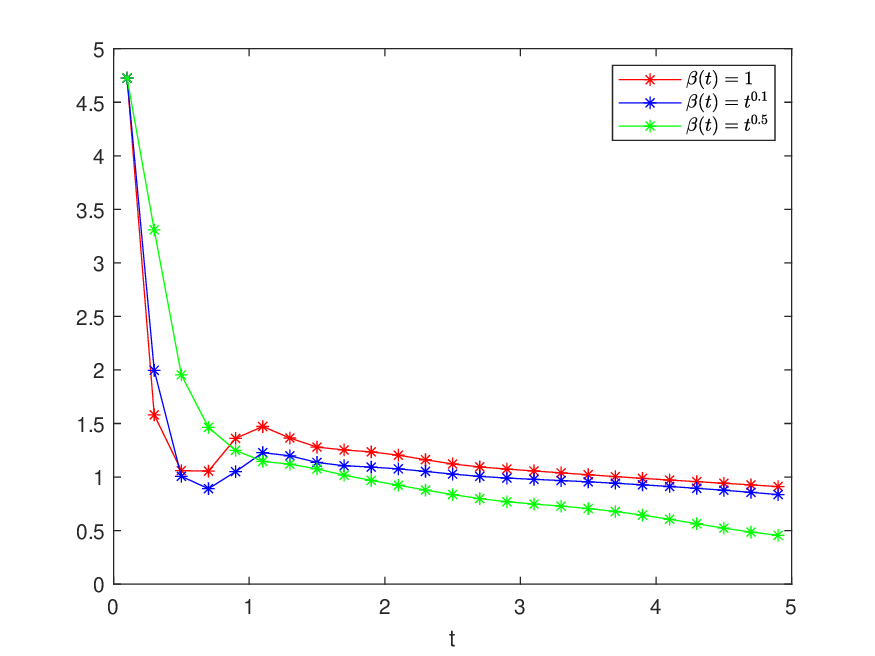}
\end{subfigure}
\hfill
\begin{subfigure}[t]{0.45\textwidth}
\rotatebox{90}{\scriptsize{~~~~~~~~~~~~~$\|Ax(t)+By(t)-b\|$}}
\hspace{-1mm}
\includegraphics[width=1\linewidth]{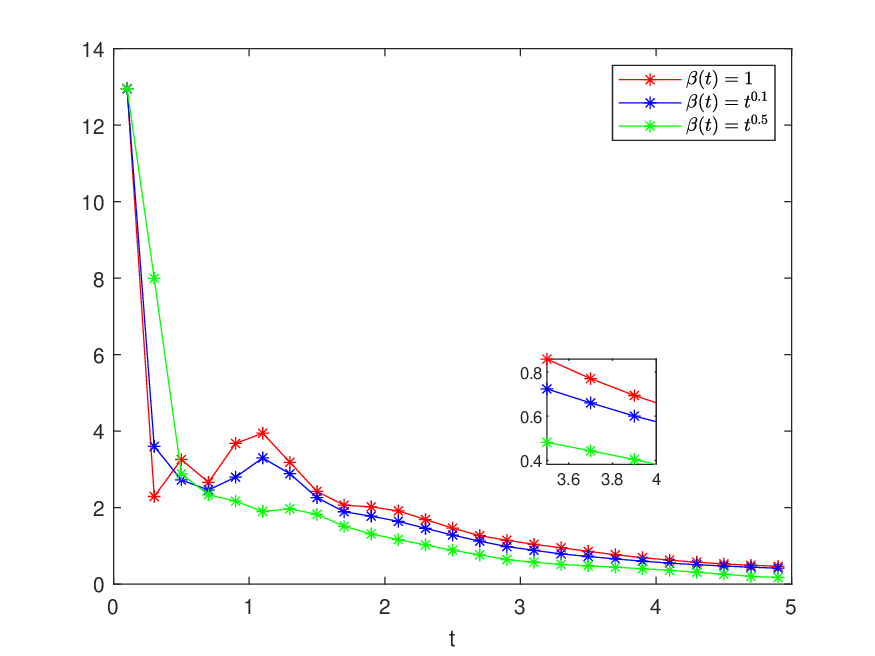}
\end{subfigure}
\begin{minipage}[t]{0.45\textwidth}
        \centering
        \vspace{-3mm}
    \end{minipage}
    \hfill
    \begin{minipage}[t]{0.45\textwidth}
        \centering
        \vspace{-3mm}
    \end{minipage}
\caption{The relative error and residual error of dynamical system with $m=15$ and $n=10$. }
\label{fig6}
\end{figure}
\begin{figure}[h]
\centering
\setlength{\abovecaptionskip}{0.cm}
\begin{subfigure}[t]{0.45\textwidth}
\rotatebox{90}{\scriptsize{~~~~~~~~~~~~$\frac{\|(x(t),y(t))-(x^*,y^*)\|}{\|(x^*,y^*)\|}$  }}
\hspace{-1mm}
\includegraphics[width=1\linewidth]{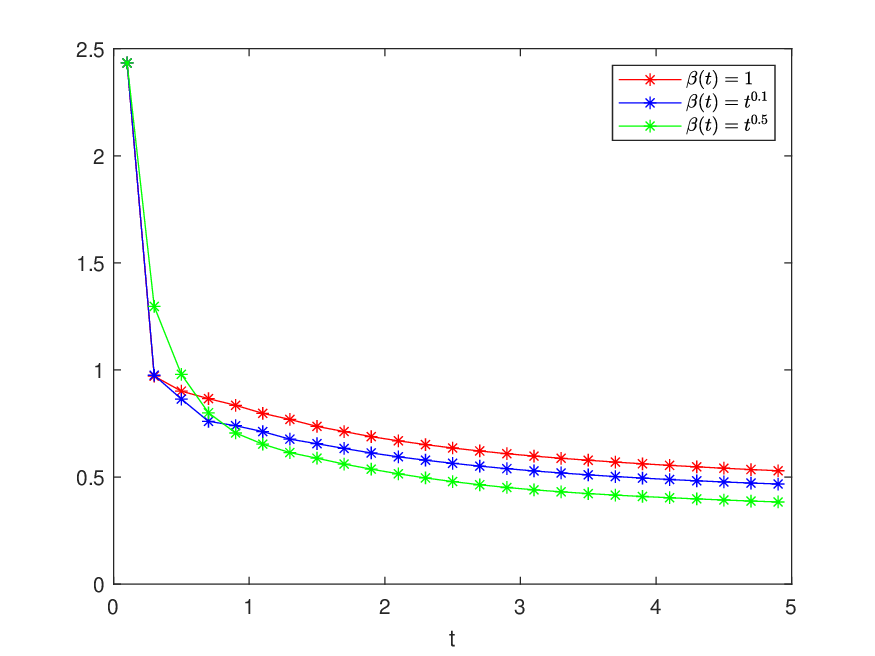}
\end{subfigure}
\hfill
\begin{subfigure}[t]{0.45\textwidth}
\rotatebox{90}{\scriptsize{~~~~~~~~~~~~~$\|Ax(t)+By(t)-b\|$}}
\hspace{-1mm}
\includegraphics[width=1\linewidth]{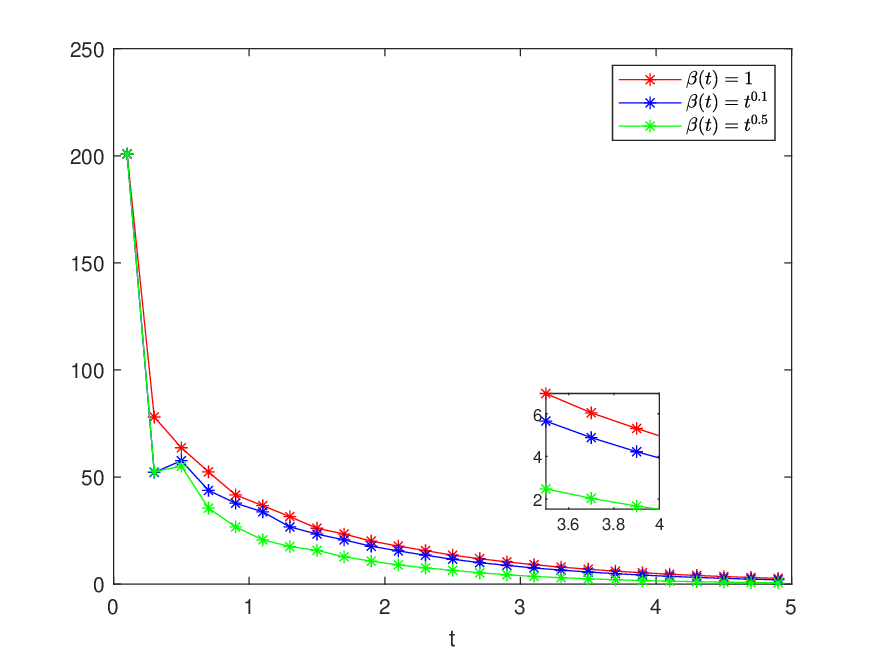}
\end{subfigure}
\begin{minipage}[t]{0.45\textwidth}
        \centering
        \vspace{-3mm}
    \end{minipage}
    \hfill
    \begin{minipage}[t]{0.45\textwidth}
        \centering
        \vspace{-3mm}
    \end{minipage}
\caption{The relative error and residual error of dynamical system with $m=120$ and $n=100$. }
\label{fig7}
\end{figure}

As shown in Figures \ref{fig6} and  \ref{fig7},  the relative  and residual errors of the dynamical system \eqref{nonsmooth} exhibit improved performance when  a faster-growing time scaling parameter $\beta(t)$ is utilized.

\section{Conclusions}
In this paper, we introduce a second-order plus first-order inertial primal-dual dynamical system with time scaling and Tikhonov regularization for the separable convex optimization problem (\ref{1.1}). Compared with the dynamical systems introduced in \cite{hhf2021,acfr2022} for the separable convex optimization problem (\ref{1.1}), the dynamical system (\ref{1.3}) uses second-order ordinary differential equations for the primal variables and first-order ordinary differential equation for the dual variable. We impose some mild assumptions and use Lyapunov analysis to show that the convergence rate of the primal-dual gap along the trajectory is $\mathcal{O}\left(\frac{1}{\beta(t)}\right)$. We also show that the convergence rates of the feasibility measure, objective error along  and the gradient norm of the objective function along the trajectory are all of $\mathcal{O}\left(\frac{1}{\sqrt{\beta(t)}}\right)$. In addition, we show that the primal trajectory converges strongly to the minimal norm solution of the separable convex optimization problem (\ref{1.1}). Through  numerical experiments, we observe that the dynamical system (\ref{1.3}) has fast convergence rates for the objective function error and the feasibility measure.

Although some new results have been obtained on the primal-dual dynamical system for separable convex optimization problems, there are remaining questions to be addressed in the future. For instance, following \cite{zdx2024,abcr2023}, can we formulate numerical algorithms to study the theoretical convergence properties of this primal-dual dynamical system in terms of time discretization? On the other hand, it is also important to consider the primal-dual dynamical system (\ref{1.3}) with asymptotically vanishing damping in the future. Furthermore, can we obtain some convergence results for the dynamical system (\ref{1.3}) with  Hessian damping?
\section*{Funding}
{\small This research is supported by the Natural Science Foundation of Chongqing (CSTB2024NSCQ-MSX0651) and the Team Building Project for Graduate Tutors in Chongqing (yds223010).}

\section*{Acknowledgements}
{\small The authors would like to express their sincere thanks to the two anonymous reviewers for the constructive suggestions and comments.}

\section*{Declaration}
{\small $\mathbf{Conflict ~of~ interest}$ No potential conflict of interest was reported by the authors.}
\bibliographystyle{plain}

\begin{thebibliography}{99}
\makeatletter
\def\@biblabel#1{#1.}
\makeatother
{\small


\bibitem{2024gopt} Adly, S., Attouch, H.: Accelerated optimization through time-scale analysis of inertial dynamics with asymptotic vanishing and Hessian-driven dampings. Optimization. (2024) \textcolor{blue}{https://doi.org/10.1080/02331934.2024.2359540}

\bibitem{ACS2021} Alecsa, C.D., L\'aszl\'o, S.C.: Tikhonov regularization of a perturbed heavy ball system with vanishing damping. SIAM J. Optim. \textbf{31}: 2921--2954 (2021)

\bibitem{abcr2023} Attouch, H., Balhag, A., Chbani, Z., Riahi, H.: Accelerated gradient methods combining Tikhonov regularization with geometric damping driven by the Hessian. Appl. Math. Optim. \textbf{88}: 29 (2023)

\bibitem{bot2023} Attouch, H., Bo\c{t}, R.I.,  Csetnek, E.R.: Fast optimization via inertial dynamics with closed-loop damping.  J. Eur. Math. Soc. \textbf{25}: 1985--2056 (2023)

\bibitem{ac2017} Attouch, H., Cabot, A.: Asymptotic stabilization of inertial gradient dynamics with time-dependent viscosity. J. Differ. Equ. \textbf{263}: 5412--5458 (2017)

\bibitem{accr2018} Attouch, H., Cabot, A., Chbani, Z., Riahi, H.: Rate of convergence of inertial gradient dynamics with time-dependent viscous
damping coefficient. Evol. Equ. Control The. \textbf{7}: 353--371 (2018)

\bibitem{mp2018att} Attouch, H., Chbani, Z., Peypouquet, J., Redont, P.: Fast convergence of inertial dynamics and algorithms with asymptotic vanishing viscosity. Math. Program. \textbf{168}: 123--175 (2018)

\bibitem{ACr2018} Attouch, H., Chbani, Z., Riahi, H.: Combining fast inertial dynamics for convex optimization  with Tikhonov regularization. J. Math. Anal. Appl. \textbf{457}: 1065--1094 (2018)

\bibitem{acr2019} Attouch, H., Chbani, Z., Riahi, H.: Fast proximal methords via time scaling of damped inertial dynamics. SIAM J. Optim. \textbf{29}: 2227--2256 (2019)

\bibitem{acfr2022} Attouch, H., Chbani, Z., Fadili, J., Riahi, H.: Fast convergence of dynamical ADMM via time scaling of damped inertial
    dynamics. J. Optim. Theory Appl. \textbf{193}: 704-736 (2022)

\bibitem{baijc} Bai, J.C., Li, J.C., Xu, F.M., Zhang, H.C.: Generalized symmetric ADMM for separable convex programming. Comput. Optim. Appl. \textbf{70}: 129--170 (2018)

\bibitem{bcl2021} Bo\c{t}, R.I., Csetnek, E.R., L\'aszl\'o, S.C.: Tikhonov regularization of a second-order dynamical system with hessian driven damping. Math. Program. \textbf{189}: 151--186 (2021)

\bibitem{jmaa2024bot} Bo\c{t}, R.I., Csetnek, E.R., L\'aszl\'o, S.C.: On the strong convergence of continuous Newton-like inertial dynamics with Tikhonov regularization for monotone inclusions. J. Math. Anal. Appl. \textbf{530}:  127689 (2024)

\bibitem{botc2022} Bo\c{t}, R.I., Karapetyants, M.A.: A fast continuous time approach with time scaling for nonsmooth convex optimization. Adv. Cont. Discr. Mod. 2022: \textbf{73} (2022)

\bibitem{BN2021} Bo\c{t}, R.I., Nguyen, D.K.: Improved convergence rates and trajectory convergence for primal-dual dynamical systems with vanishing damping. J. Differ. Equ. \textbf{303}: 369--406 (2021)

\bibitem{bpcpe20113} Boyd, S., Parikh, N., Chu, E., Peleato, B., Eckstein, J.: Distributed optimization and statistical learning via the alternating direction method of multipliers. Found. Trends Mach. Learn. \textbf{3}: 1--122 (2011)

\bibitem{CEG2009} Cabot, A., Engler, H., Gadat, S.: On the long time behavior of second-order differentila equations with asymtotically small dissipation. Trans. Amer. Math. Soc. \textbf{361}: 5983--6017 (2009)

\bibitem{ceg2009} Cabot, A., Engler, H., Gadat, S.: Second-order differential equations with asymptotically small dissipation and piecewise at potentials. Electron.  J. Differ. Equ. \textbf{17}: 33--38 (2009)

\bibitem{ck2024} Csetnek, E.R., Karapetyants,  M.A.: Second-order dynamics featuring Tikhonov regularization and time scaling. J Optim Theory Appl.  \textbf{202}: 1385--1240 (2024)

\bibitem{fg1983} Fortin, M., Glowinski, R.: Augmented Lagrangian methods. Elsevier (1983)

\bibitem{gosr20142} Goldstein, T., O'Donoghue, B., Setzer, S., Baraniuk, R.: Fast alternating direction optimization methods. SIAM J. Imag. Sci. \textbf{7}: 1588--1623 (2014)

\bibitem{ah1991} Haraux, A.: Syst\`emes dynamiques dissipatifs et applications. Masson, Paris (1991)

\bibitem{hhf2021} He, X., Hu, R., Fang,Y.P.: Convergence rates of inertial primal-dual dynamical methods for separable convex optimization problems. SIAM J. Control Optim. \textbf{59}: 3278--3301 (2021)

\bibitem{hhf2022} He, X., Hu, R., Fang, Y.P.: Fast primal-dual algorithm via dynamical system for a linearly constrained convex optimization problem. Automatica  \textbf{146}: 110547 (2022)

\bibitem{HHF2022} He, X., Hu, R., Fang, Y.P.: Second-order primal + first-order dual dynamical systems with time scaling for linear equality constrained convex optimization problems. IEEE Trans. Autom. Control. \textbf{67}: 4377--4383 (2022)

\bibitem{hhf2023} He, X., Hu, R., Fang, Y.P.: Inertial primal-dual dynamics with damping and scaling for linearly constrained convex optimization problems. 	Appl. Anal. \textbf{102}: 4114--4139 (2023)

\bibitem{hhf2024} He, X., Hu, R., Fang, Y.P.: A second-order primal-dual dynamical system for a convex-concave bilinear saddle point problem. Appl. Math. Optim. \textbf{89}: 30 (2024)

\bibitem{htlf2023} He, X., Tian, F., Li, A., Fang, Y.P.: Convergence rates of mixed primal-dual dynamical systems with Hessian driven damping. Optimization \textbf{74}: 365--390 (2023).

\bibitem{hn2023} Hulett, D.A., Nguyen, D.K.: Time rescaling of a primal-dual dynamical system with asymptotically vanishing damping. Appl. Math. Optim. \textbf{88}: 27 (2023)

\bibitem{zdx2024} Jiang, Z.Y., Wang, D., Liu, X.W.: A fast primal-dual algorithm via dynamical system with variable mass for linearly constrained convex optimization. Optim. Lett. \textbf{18}: 1855--1880 (2024)

\bibitem{ka2024} Karapetyants, M.A.: A fast continuous time approach for non-smooth convex optimization using Tikhonov regularization technique. Comput. Optim. Appl. \textbf{87}: 531--569 (2024)

\bibitem{L2023} L\'aszl\'o, S.C.: On the strong convergence of the trajectories of a Tikhonov regularized second-order dynamical system with asymptotically vanishing damping. J. Differ. Equ. \textbf{362}: 355--381 (2023)

\bibitem{LS1997} Lemar\'{e}chal, C., Sagastiz\'{a}bal, C.: Practical aspects of the Moreau-Yosida regularization: theoretical preliminaries. SIAM J. Optim. \textbf{7}: 367--385 (1997)

\bibitem{llf20194} Lin, Z., Li, H., Fang, C.: Accelerated optimization for machine learning. Springer Nature (2019)

\bibitem{y1983} Nesterov, Y.: A method of solving a convex programming problem with convergence rate $\mathcal{O}\left(\frac{1}{k^2}\right)$. Insov. Math. Dokl. \textbf{27}: 372--376 (1983)


\bibitem{n2004} Nesterov, Y.: Introductory lectures on convex optimization. Applied Optimization. Springer (2004)

\bibitem{B1964} Polyak, B.T.: Some methods of speeding up the convergence of iteration methods. Comput. Math. Math. Phys. \textbf{4}: 1--17 (1964)

\bibitem{SBC2016} Su, W., Boyd, S., Cand\`es, E.: A differential equation for modeling Nesterov's accelerated gradient method. J. Mach. Learn. Res. \textbf{17}: 5312--5354 (2016)

\bibitem{wwj2016} Wibisono, A., Wilson, A.C., Jordan, M.I.: A variational perspective on accelerated methods in optimization. Proc. Natl. Acad. Sci. \textbf{113}: E7351--E7358 (2016)

\bibitem{wrj2021} Wilson, A.C., Recht, B., Jordan, M.I.: A Lyapunov analysis of accelerated methods in optimization. Mach. Learn. Res. \textbf{22}: 1--34 (2021)

\bibitem{xw2021} Xu, B., Wen, B.: On the convergence of a class of inertial dynamical systems with Tikhonov regularization. Optim. Lett. \textbf{15}: 2025--2052 (2021)

\bibitem{hexing} Zhao, Y., He, X., Zhou, M.L., Huang, T.W.: Accelerated primal-dual projection neurodynamic approach with time scaling for linear and set constrained convex optimization problems. IEEE/CAA J. Autom. Sinica. \textbf{11}: 1485--1498 (2024)


\bibitem{ZLC2023} Zeng, X.L., Lei, J., Chen, J.: Dynamical primal-dual nesterov accelerated method and its application to network optimization. IEEE Trans. Autom. Control. \textbf{68}: 1760--1767 (2023)

\bibitem{zhf2024} Zhu, T.T., Hu, R., Fang, Y. P.: Fast convergence rates and trajectory convergence of a Tikhonov regularized inertial primal\mbox {-}dual dynamical system with time scaling and vanishing damping. J. Comput. Appl. Math. \textbf{460}: 116394 (2024)

\bibitem{ZHF2024} Zhu, T.T., Hu, R., Fang, Y.P.: Tikhonov regularized second-order plus first-order primal-dual dynamical systems with asymptotically vanishing damping for linear equality constrained convex optimization problems. Optimization. (2024) \textcolor{blue}{https://doi.org/10.1080/02331934.2024.2407515}
}
\end{thebibliography}

\end{document}